\UseRawInputEncoding
\documentclass[12pt,reqno]{amsart}

\makeatletter
\def\@settitle{\begin{center}%
  \baselineskip14\p@\relax
  \bfseries
  \@title
  \end{center}%
}

\makeatother

\usepackage{amsmath,amsfonts,amsthm,amsopn,amssymb,extarrows,multirow,footnote,float}
\usepackage{cite,marginnote}
\usepackage{booktabs}
\usepackage{bm}
\usepackage{color,enumitem,graphicx}
\usepackage[colorlinks=true,urlcolor=blue,
citecolor=blue,linkcolor=blue,linktocpage,pdfpagelabels,
bookmarksnumbered,bookmarksopen]{hyperref}
\usepackage[english]{babel}

\usepackage[left=2.9cm,right=2.9cm,top=2.8cm,bottom=2.8cm]{geometry}
\usepackage[hyperpageref]{backref}

\numberwithin{equation}{section}

\makeindex

\def\N{\mathbb{N}}

\newtheorem{theorem}{Theorem}[section]

\newtheorem{lemma}[theorem]{Lemma}
\newtheorem{corollary}[theorem]{Corollary}
\newtheorem{proposition}[theorem]{Proposition}
\newtheorem{remark}[theorem]{Remark}

\newcommand{\R}{\mathbb R}

\newcommand{\bt}{\begin{theorem}}
	\newcommand{\et}{\end{theorem}}
\newcommand{\bl}{\begin{lemma}}
	\newcommand{\el}{\end{lemma}}
	\newcommand{\ed}{\end{definition}}
\newcommand{\bc}{\begin{corollary}}
	\newcommand{\ec}{\end{corollary}}
\newcommand{\bp}{\begin{proof}}
	\newcommand{\ep}{\end{proof}}
\newcommand{\bx}{\begin{example}}
	\newcommand{\ex}{\end{example}}
\newcommand{\bi}{\begin{exercise}}
	\newcommand{\ei}{\end{exercise}}
\newcommand{\bo}{\begin{proposition}}
	\newcommand{\eo}{\end{proposition}}
\newcommand{\br}{\begin{remark}}
	\newcommand{\er}{\end{remark}}
\newcommand{\beq}{\begin{equation}}
	\newcommand{\eeq}{\end{equation}}
\newcommand{\ba}{\begin{align}}
	\newcommand{\ea}{\end{align}}
\newcommand{\bn}{\begin{enumerate}}
	\newcommand{\en}{\end{enumerate}}
\newcommand{\bg}{\begin{align*}}
	\newcommand{\bcs}{\begin{cases}}
		\newcommand{\ecs}{\end{cases}}

	\newcommand{\bean}{\begin{eqnarray*}}
		\newcommand{\eean}{\end{eqnarray*}}

	\def\N{\mathbb{N}}

	\def\R{\mathbb{R}}

\def\@setauthors{\begin{center}%
  \def\and{\unskip\\}%
  \baselineskip16\p@\relax
  \normalfont
  \@author
  \end{center}%
}
	\renewcommand{\abstractname}{\normalfont\bfseries Abstract}

	\title[NLS with partial confinement]{Liouville theorems and symmetry of positive solutions for partially confined nonlinear Schr\"odinger equations}
	
	\author[Z.~J.~Chen]{Zhijie Chen}
    \author[Y.~J.~Wang]{Youjun Wang}
	\author[J.~J.~Zhang]{Jianjun Zhang}
	\author[X.~X.~Zhong]{Xuexiu Zhong}

\address[Z.~J.~Chen]{\newline\indent Department of Mathematical Sciences \& Yau Mathematical Sciences Center
\newline\indent
Tsinghua University
\newline\indent
 Beijing, 100084, China}
\email{\href{mailto:zjchen2016@tsinghua.edu.cn}{zjchen2016@tsinghua.edu.cn}}

\address[Y.~J.~Wang]{\newline\indent Department of Mathematical Sciences
\newline\indent
South China University of Technology
\newline\indent
Guangzhou, 510640, China}
\email{\href{mailto:scyjwang@scut.edu.cn}{scyjwang@scut.edu.cn}}

\address[J.~J.~Zhang]{\newline\indent College of Mathematics and Statistics
\newline\indent
Chongqing Jiaotong University
\newline\indent
Chongqing, 400074, China}
\email{\href{mailto:zhangjianjun09@tsinghua.org.cn}{zhangjianjun09@tsinghua.org.cn}}

	\address[X.~X.~Zhong]{\newline\indent South China Research Center for Applied Mathematics and Interdisciplinary Studies \& School of Mathematical Sciences
		\newline\indent
		South China Normal University
		\newline\indent
		Guangzhou, 510631, China}
	 \email{\href{mailto:zhongxuexiu1989@163.com}{zhongxuexiu1989@163.com}}

\makeatletter
\@namedef{subjclassname@2020}{%
  \textup{2020} Mathematics Subject Classification}
\makeatother

\date{\today}

\begin{document}
\renewcommand{\abstractname}{\normalfont\bfseries Abstract}
\begin{abstract}
We study positive solutions of the partially confined stationary nonlinear Schr\"odinger equation
$$
 -\Delta u+|y|^2u+\lambda u=g(u),
 \quad (y,z)\in\mathbb{R}^d\times\mathbb{R}^{m},
 \quad 1\le d<N,\quad m:=N-d.
$$
The spectrum of $-\Delta + |y|^2$ is given by $[d,\infty)$. We prove that
 $\lambda\ge-d$ is necessary for positive solutions in the classes considered here,
 and  the threshold $\lambda=-d$ yields several Liouville type results.  In particular, for the pure-power equation
$$
 (-\Delta + |y|^2-d)u=u^p,
$$
we prove that any nonnegative $H^1(\mathbb{R}^N)$ weak solution is trivial whenever
$1\leq p\leq \max\{(N+2)/(N-2), m/(m-2)\}$ for $m\geq 3$ or $p\geq 1$ for $m=1,2$.  The proof combines a half-space oscillator gap, moving planes at the spectral threshold, a Picone inequality, a strict Gaussian second-moment inequality, and anisotropic Pohozaev identities.   For $\lambda>-d$, we establish the existence of positive solutions under some standard assumptions.  Furthermore, every positive solution decaying at infinity is radially symmetric and strictly decreasing in the confined variables and, up to one common translation, radially symmetric and strictly decreasing in the free variables.

\vskip 0.2in
Dedicated to our supervisor Prof. Wenming Zou on the occasion of his 60th birthday.
\end{abstract}

\maketitle

\section{Introduction}
\label{sec:introduction}

This paper is devoted to positive solutions of the stationary nonlinear Schr\"odinger equation with a partial harmonic confinement
\begin{equation}\label{eq:main}
 -\Delta u+|y|^2u+\lambda u=g(u),
 \qquad x=(y,z)\in\mathbb{R}^d\times\mathbb{R}^m,
 \qquad m:=N-d\geq1,
\end{equation}
where $1\leq d\leq N-1$.  We write
$$
 H:=-\Delta+|y|^2=H_y-\Delta_z,
 \qquad H_y:=-\Delta_y+|y|^2,
$$
and use the natural energy space
\begin{equation}\label{eq:20260518-0953}
 X:=\left\{u\in H^1(\mathbb{R}^N):
       \int_{\mathbb{R}^N}|y|^2u^2\,\mathrm{d}x<\infty\right\}.
\end{equation}
As usual, $2^*=2N/(N-2)$ if $N\geq3$ and $2^*=+\infty$ if $N=2$.

Equation \eqref{eq:main} arises from standing waves for nonlinear Schr\"odinger equations in anisotropic traps.  Partial confinement is particularly natural in models of elongated Bose--Einstein condensates, where the transverse variables are trapped while one or more longitudinal variables remain free; see, for instance, \cite{ACD,BBJV-2017}.  In contrast with the fully confined problem, the embedding of $X$ is not compact because translations in the $z$-variables leave both the quadratic form and the equation invariant.  This loss of compactness is one of the principal analytic features of the problem.

In the absence of partial trapping, the variational theory of bound states on
$\mathbb{R}^N$ has a long history and it seems impossible for us to list all the references.  For foundational existence and compactness
methods for scalar-field and nonlinear Schr\"odinger equations, see e.g.
 \cite{BerestyckiLions1983-I,Rabinowitz,CR,BW} and the references therein.  For the existence of multiple solutions and nodal
solutions, see e.g.
\cite{BerestyckiLions1983-II,BW1} and the references therein. For minimax methods for
indefinite Schr\"odinger functionals, see e.g. \cite{SW} and the references therein.  For
fixed-frequency problems with critical and mixed-power nonlinearities, see
e.g. \cite{WW}  and the references therein.  These works provide the broader variational background, but
they do not address the anisotropic translation defect or the threshold
$\lambda=-d$ created by partial confinement.

For the physically important case $(N,d)=(3,2)$ and pure-power nonlinearities, Bellazzini, Boussa\"id, Jeanjean and Visciglia \cite{BBJV-2017} established existence and stability results for normalized standing waves.  Their symmetry argument is tied to the variational character of ground states.  Jeanjean and Song \cite{JS} subsequently obtained a second normalized solution in the mass-supercritical regime.  More recently, Shan, Shuai and Wu \cite{ShanShuaiWu2025} obtained multiplicity results and developed a maximum-principle approach which yields the symmetry of positive solutions for \eqref{eq:main} when $\lambda\geq0$.  The interval
$$
 -d<\lambda<0
$$
is subtler: the zero-order part is not pointwise positive, and the standard weak maximum principle used for $\lambda\geq0$ no longer applies directly.  A second delicate regime is the threshold $\lambda=-d$, because $d$ is the bottom of the spectrum of $H$ and hence $H-d$ loses the global coercivity.

The symmetry analysis belongs to the classical moving-plane tradition
\cite{GidasNiNirenberg1979,Gidas-1980,LiNi1993}.  Partial confinement,
however, requires a replacement for the usual pointwise positivity of the
zero-order coefficient; this is precisely where the spectral lower bounds
developed below enter.

The purpose of the present paper is to give a unified treatment of these two threshold phenomena.  Besides correcting several functional-analytic points that are essential in a noncompact setting, we obtain several zero-mass Liouville type results.  The main ingredients are as follows.

\begin{itemize}
 \item We realize $H$ through its closed quadratic form, distinguish the form domain $Q(H)=X$ from the operator domain $D(H)$, and use the Hermite expansion in $y$ together with the Fourier transform in $z$ to prove that
 $$
   \sigma(H)=\sigma_{\rm ac}(H)=[d,\infty),
   \qquad \sigma_{\rm p}(H)=\sigma_{\rm sc}(H)=\varnothing.
 $$
 In particular, $\lambda\geq-d$ is necessary for positive solutions in the classes considered below.

 \item At the zero-mass threshold, the projection onto the oscillator ground state
 $$
   \phi_0(y)=\pi^{-d/4}e^{-|y|^2/2},
   \qquad H_y\phi_0=d\phi_0,
 $$
 reduces suitable inequalities for $H-d$ to Liouville inequalities for $-\Delta_z$.  This gives a family of free-dimensional Liouville type results.

 \item For the model equation
 \begin{equation}\label{eq:main-p}
   (H-d)u=u^p,
   \qquad 0\leq u\in H^1(\mathbb{R}^N),
 \end{equation}
 we prove a critical finite-energy Liouville theorem without any restriction $m=N-d$: if $N\geq3$ and
 $$
   1\leq p\leq p_c:=\frac{N+2}{N-2},
 $$
 then $u\equiv0$.  The proof is not a direct consequence of the ground-state projection.  It combines a Dirichlet half-space gap of size two for the harmonic oscillator, a moving-plane argument exactly at $\lambda=-d$, an annihilation-operator/Picone argument for the quotient $u/\phi_0$, a strict Gaussian second-moment inequality, and two anisotropic Pohozaev identities.

 \item For $\lambda>-d$, we give a mountain-pass proof under standard subcritical Ambrosetti--Rabinowitz assumptions.  The noncompactness in the free variables is handled by an anisotropic Lions-type nonvanishing lemma, and the passage to the nonlinear limit is carried out using local compactness and uniform integrability.

 \item Finally, for every $\lambda>-d$ we prove symmetry for arbitrary positive finite-energy solutions satisfying the stated decay assumption.  The key point in the negative spectral gap is to combine a small-negative-set estimate with the deformed oscillator lower bound
 $$
   \inf\sigma\bigl(-(1-\varepsilon)\Delta+|y|^2\bigr)
   =d\sqrt{1-\varepsilon}.
 $$
 This yields radial symmetry and strict radial monotonicity in the confined variables, and radial symmetry and strict radial monotonicity in the free variables up to a common translation.
\end{itemize}

We next state the assumptions and main results precisely.  In the alternative growth hypothesis used below, the dependence on the free variable must be locally uniform.  We therefore use the following condition.
\begin{enumerate}[label=$\mathbf{(A)}$]
\item \label{Assumption:A}
For every compact set $K\Subset\mathbb{R}^m$, there exist constants $C_K>0$ and $\kappa_K\in[0,1/2)$ such that
\begin{equation}\label{eq2-1}
 \sup_{z\in K}\bigl(
 |u(y,z)|+|\nabla_yu(y,z)|+|\Delta_yu(y,z)|\bigr)
 \leq C_Ke^{\kappa_K|y|^2},
 \qquad y\in\mathbb{R}^d.
\end{equation}
\end{enumerate}
The local uniformity in $z$ is what permits the Gaussian-weighted Fubini and limiting arguments used in Sections \ref{sec:Spectral} and \ref{sec:Liouville-results}.

\begin{theorem}[Threshold and Liouville results]\label{thm:nec-liou}
Let $g\in C([0,\infty),[0,\infty))$.
\begin{enumerate}[label=(\roman*)]
 \item \label{necessary}
 Suppose that \eqref{eq:main} has a positive classical solution $u\in C^2(\mathbb{R}^N)$ such that either $u\in X$ or $u$ satisfies \ref{Assumption:A}.  Then necessarily
 $$
   \lambda\geq-d.
 $$

 \item \label{zero-mass}
 Let $u\in C^2(\mathbb{R}^N)$ be nonnegative, assume either $u\in X$ or \ref{Assumption:A}, and suppose $\lambda=-d$.
 \begin{enumerate}[label=(ii-\arabic*)]
  \item \label{Liouville-1}
  If $(H-d)u\geq0$ and $u\in L^q(\mathbb{R}^N)$, where $q\geq1$ if $m\leq2$ and
  $$
    1\leq q\leq\frac{m}{m-2}\quad\text{if }m>2,
  $$
then $u\equiv0$.

  \item \label{Liouville-2}
  If $(H-d)u\geq u^p$, where $p\geq1$ if $m\leq2$ and
  $$
    1\leq p\leq\frac{m}{m-2}\quad\text{if }m>2,
  $$
then $u\equiv0$.
 \end{enumerate}

 \item \label{application}
 The following finite-energy consequences hold.
 \begin{enumerate}[label=(iii-\arabic*)]
  \item \label{Liouvilleiii-1}
  If $m\leq4$ and $0\leq u\in C^2(\mathbb{R}^N)\cap X$ solves
  $$
    (H-d)u=g(u),
  $$
  then $u\equiv0$.

  \item \label{Liouvilleiii-2}
  If $N\geq3$ and $1\leq p\leq (N+2)/(N-2)$, then the pure-power problem \eqref{eq:main-p} has only the trivial nonnegative $H^1(\mathbb{R}^N)$ weak solution.  When $N=2$, the same conclusion holds for every finite $p\geq1$.

  \item \label{Liouvilleiii-3}
  Let $m>2$ and $p>1$.  If $0\leq u\in C^2(\mathbb{R}^N)\cap X\cap L^{p+1}(\mathbb{R}^N)$ is a nontrivial solution of \eqref{eq:main-p}, then necessarily
  $$
    p<\frac{m+2}{m-2}.
  $$
 \end{enumerate}
\end{enumerate}
\end{theorem}

For $\lambda>-d$, we use the following standard hypotheses:
\begin{enumerate}[label=$\mathbf{(G{\arabic*})}$]
 \item \label{gcon-1}
 $g\not\equiv0$, $\displaystyle\lim_{s\rightarrow0^+}g(s)/s=0$, and there exists $q\in(2,2^*)$ such that
 $$
   \limsup_{s\rightarrow\infty}\frac{g(s)}{s^{q-1}}<\infty.
 $$
 \item \label{gcon-2}
 There exists $\theta>2$ such that
 $$
   0\leq\theta G(s)\leq g(s)s,
   \qquad G(s):=\int_0^s g(t)\,\mathrm{d}t.
 $$
\end{enumerate}

\begin{theorem}[Existence]\label{thm:existence}
Assume \ref{gcon-1}--\ref{gcon-2}.  Then, for every $\lambda>-d$, equation \eqref{eq:main} possesses a positive solution
$$
 u\in X\cap L^\infty(\mathbb{R}^N)\cap C^{1,\alpha}_{\rm loc}(\mathbb{R}^N)
 \quad\text{for every }\alpha\in(0,1),
$$
satisfying $u(x)\rightarrow0$ as $|x|\rightarrow\infty$.  If, in addition, $g\in C^1$, then
$u\in C^{2,\alpha}_{\rm loc}(\mathbb{R}^N)$ for every $\alpha\in(0,1)$.
\end{theorem}

\begin{theorem}[Symmetry]\label{thm:main-symmetry}
Assume $\lambda>-d$ and $g\in C^1([0,\infty),[0,\infty))$ with $g'(0)=0$.  Let
$u\in C^2(\mathbb{R}^N)\cap X$ be a positive solution of \eqref{eq:main} such that $u(x)\rightarrow0$ as $|x|\rightarrow\infty$.  Then there exist $z_0\in\mathbb{R}^m$ and a function $U=U(r,s)$ such that
$$
 u(y,z)=U(|y|,|z-z_0|).
$$
Moreover,
$$
 \partial_rU(r,s)<0\quad(r>0),
 \qquad
 \partial_sU(r,s)<0\quad(s>0).
$$
In particular, $u$ is radially symmetric about the origin in the confined variables and, up to a translation, radially symmetric in the free variables.
\end{theorem}

Of course, one may give some assumptions on $g(u)$ weaker than \ref{gcon-1}-\ref{gcon-2} (for example, $g(u)$ might be of critical growth) to guarantee the existence of positive solutions. In this paper, we state Theorem \ref{thm:existence}, as a supplement of Theorem \ref{thm:nec-liou}, just to indicate that $\lambda>-d$ is sufficient for the existence of positive solutions. Therefore, we do not want to perform an in-depth investigation of the conditions on $g(u)$ here.

Theorem \ref{thm:nec-liou}\,(iii-2) settles the critical \emph{nonexistence direction} of the pure-power zero-mass problem, but it does not by itself identify a sharp existence threshold.  Indeed, when $m\geq5$, the two lower nonexistence exponents satisfy
\begin{equation}\label{eq:intro-ordering}
 \frac{N+2}{N-2}\geq\frac{m}{m-2}
 \quad\Longleftrightarrow\quad m\geq d+2,
\end{equation}
while Theorem \ref{thm:nec-liou}\,(iii-3) gives the upper obstruction $p<(m+2)/(m-2)$ for every nontrivial finite-energy solution.  Thus the remaining possible interval is contained in
\begin{equation}\label{eq:intro-window}
 \max\left\{\frac{m}{m-2},\frac{N+2}{N-2}\right\}
 <p<\frac{m+2}{m-2}.
\end{equation}
No existence assertion in \eqref{eq:intro-window} is made here.

\begin{remark}[Remaining questions]\label{remark:open-problems}
The preceding results leave several natural problems.
\begin{itemize}
 \item For general nonlinearities $g(s)=o(s)$ at zero, determine whether the zero-mass equation $(H-d)u=g(u)$ can possess positive finite-energy solutions and, if so, whether all such solutions inherit the same symmetry as in Theorem \ref{thm:main-symmetry}.
 \item For the pure-power zero-mass problem with $m\geq5$, determine whether nontrivial finite-energy solutions exist anywhere in the interval \eqref{eq:intro-window}; this is the remaining sharpness problem after Theorem \ref{thm:nec-liou}\,(iii-2).
 \item For $\lambda>-d$ and pure powers, study uniqueness and nondegeneracy of positive solutions modulo translations in the free variables.
\end{itemize}
\end{remark}

The paper is organized as follows.  Section \ref{sec:Spectral} gives the form realization and the Hermite--Fourier spectral decomposition, and proves Theorem \ref{thm:nec-liou}\,(i).  Section \ref{sec:Liouville-results} proves the projected Liouville theorems and the critical finite-energy zero-mass theorem, including the half-space gap, annihilation-operator inequality and anisotropic Pohozaev identities.  Section \ref{sec:existence} establishes Theorem \ref{thm:existence} with a complete concentration-compactness argument.  Section \ref{sec:symmetry} proves Theorem \ref{thm:main-symmetry}.

\noindent\textbf{Notation.}
For $u\in H^1$, write $u^+=\max\{u,0\}$ and $u^-=\min\{u,0\}$.  The letter $C$ denotes a positive constant whose value may change from line to line.  All inequalities involving $H$ are understood through its closed quadratic form unless an operator-domain statement is explicitly made.

\section{\texorpdfstring{Spectral theory and the threshold $\lambda=-d$}{Spectral theory and the threshold lambda=-d}}\label{sec:Spectral}

We begin by fixing the functional-analytic realization of the partially confined Schr\"odinger operator.  Consider the densely defined closed quadratic form
\begin{equation}\label{eq:spectral-form}
 \mathfrak{h}[u,v]
 :=\int_{\mathbb{R}^N}\bigl(\nabla u\cdot\nabla\overline v+|y|^2u\overline v\bigr)\,\mathrm{d}x,
 \qquad Q(\mathfrak{h})=X.
\end{equation}
The self-adjoint operator associated with \eqref{eq:spectral-form}, again denoted by $H$, has domain
\begin{equation}\label{eq:operator-domain-H}
 D(H)=\left\{u\in X:
 -\Delta u+|y|^2u\in L^2(\mathbb{R}^N)\ \text{in }\mathcal D'(\mathbb{R}^N)\right\}.
\end{equation}
This is the standard realization furnished by the first representation theorem;
see, for example, \cite[Chapter VIII]{Reed-Simon-1980}.
Thus $Q(H)=X$ is the form domain, whereas $D(H)$ is the operator domain.  The distinction will be used repeatedly below.

Let $\{h_k\}_{k\geq0}$ be the normalized one-dimensional Hermite functions,
$$
 \left(-\frac{\mathrm{d}^2}{\mathrm{d}t^2}+t^2\right)h_k=(2k+1)h_k.
$$
For a multi-index $\alpha=(\alpha_1,\ldots,\alpha_d)\in\mathbb N_0^d$, put
$$
 \phi_\alpha(y):=\prod_{j=1}^d h_{\alpha_j}(y_j),
 \qquad |\alpha|:=\alpha_1+\cdots+\alpha_d.
$$
Then $\{\phi_\alpha\}_{\alpha\in\mathbb N_0^d}$ is an orthonormal basis of $L^2(\mathbb{R}^d)$ and
\begin{equation}\label{eq:Hermite-eigenvalues}
 H_y\phi_\alpha=(d+2|\alpha|)\phi_\alpha,
 \qquad H_y:=-\Delta_y+|y|^2.
\end{equation}
In particular, the normalized ground state is
\begin{equation}\label{eq:ground-state}
 \phi_0(y)=\pi^{-d/4}e^{-|y|^2/2},
 \qquad H_y\phi_0=d\phi_0.
\end{equation}

For $u\in L^2(\mathbb{R}^N)$, we define the Hermite coefficients
$$
 u_\alpha(z):=\int_{\mathbb{R}^d}u(y,z)\overline{\phi_\alpha(y)}\,\mathrm{d}y
$$
and then take the Fourier transform in $z$.  The resulting unitary map
$$
 \mathcal U:L^2(\mathbb{R}^N)
 \longrightarrow\bigoplus_{\alpha\in\mathbb N_0^d}L^2(\mathbb{R}^m_\xi)
$$
turns $H$ into the direct sum of multiplication operators
\begin{equation}\label{eq:direct-sum-spectral}
 \mathcal U H\mathcal U^{-1}
 =\bigoplus_{\alpha\in\mathbb N_0^d}
 M_{d+2|\alpha|+|\xi|^2}.
\end{equation}
See e.g. \cite[Section 2.5]{Teschl2009}.
This immediately gives the complete spectral description.

\begin{proposition}\label{pro:20260514-1010}
The operator $H$ has purely absolutely continuous spectrum and
\begin{equation}\label{eq:20260514-0958}
 \sigma(H)=\sigma_{\rm ac}(H)=[d,\infty),
 \qquad
 \sigma_{\rm p}(H)=\sigma_{\rm sc}(H)=\varnothing.
\end{equation}
Consequently,
\begin{equation}\label{eq4-1}
 \int_{\mathbb{R}^N}\bigl(|\nabla u|^2+|y|^2u^2\bigr)\,\mathrm{d}x
 \geq d\int_{\mathbb{R}^N}u^2\,\mathrm{d}x,
 \qquad u\in X.
\end{equation}
\end{proposition}

\begin{proof} This result is well-known to experts in this field; see e.g. \cite[Corollary, Page 301]{Reed-Simon-1980} or \cite[Section 7]{Teschl2009}. Here we provide a proof for the reader's convenience.
For brevity, set
$$
 a_\alpha:=d+2|\alpha|,
 \qquad
 q_\alpha(\xi):=a_\alpha+|\xi|^2.
$$
The spectral set follows from \eqref{eq:direct-sum-spectral}: the essential range of the $\alpha$-th multiplier is $[d+2|\alpha|,\infty)$, and the union over $\alpha$ is $[d,\infty)$.
Indeed, the spectrum of a multiplication operator $M_q$ is the essential
range of its multiplier $q$.  Hence
$$
 \sigma(M_{q_\alpha})=[a_\alpha,\infty),
$$
and the spectrum of the orthogonal direct sum in
\eqref{eq:direct-sum-spectral} is the closure of the union of these sets.
Since the mode $\alpha=0$ is present and $a_0=d$, this union is already the
closed interval $[d,\infty)$.

To identify the spectral type, let $F=(F_\alpha)_\alpha$ belong to the direct sum in \eqref{eq:direct-sum-spectral}.

For a Borel set $B\subset\mathbb{R}$, the spectral projection of the
$\alpha$-th multiplication operator is multiplication by
$\mathbf 1_B(q_\alpha)$.  Therefore the spectral projection of the direct
sum acts componentwise as
$$
 \bigl(E(B)F\bigr)_\alpha(\xi)
 =\mathbf 1_B\bigl(q_\alpha(\xi)\bigr)F_\alpha(\xi).
$$
Taking its inner product with $F$ gives the following scalar spectral measure:
$$
 \mu_F(B)=\sum_\alpha\int_{\mathbb{R}^m}
 \mathbf 1_B\bigl(d+2|\alpha|+|\xi|^2\bigr)|F_\alpha(\xi)|^2\,\mathrm{d}\xi.
$$
For each fixed $\alpha$, denote the corresponding summand by
$$
 \mu_{F,\alpha}(B)
 :=\int_{\mathbb{R}^m}\mathbf 1_B(a_\alpha+|\xi|^2)
 |F_\alpha(\xi)|^2\,\mathrm d\xi.
$$
For each fixed $\alpha$, polar coordinates and the change of variables
$t=d+2|\alpha|+r^2$ show that the corresponding summand is absolutely
continuous with respect to one-dimensional Lebesgue measure.
More explicitly, when $m\geq2$, polar coordinates $\xi=r\omega$ give
$$
 \mu_{F,\alpha}(B)
 =\int_0^\infty \mathbf 1_B(a_\alpha+r^2)
 \left(\int_{\mathbb S^{m-1}}|F_\alpha(r\omega)|^2\,\mathrm d\omega\right)
 r^{m-1}\,\mathrm dr.
$$
Under the substitution $t=a_\alpha+r^2$, one has
$$
 \mathrm dt=2r\,\mathrm dr,
 \qquad
 r^{m-1}\,\mathrm dr
 =\frac12(t-a_\alpha)^{\frac{m-2}{2}}\,\mathrm dt.
$$
Consequently $\mu_{F,\alpha}(B)$ is the integral over $B$ of the density
written below.
If $m\geq2$, then, for $t>d+2|\alpha|$, its density is represented almost everywhere by
$$
 \rho_{F,\alpha}(t)
 =\frac12\bigl(t-d-2|\alpha|\bigr)^{\frac{m-2}{2}}
 \int_{\mathbb S^{m-1}}
 |F_\alpha(\sqrt{t-d-2|\alpha|}\,\omega)|^2\,\mathrm{d}\omega,
$$
and $\rho_{F,\alpha}(t)=0$ for $t\leq d+2|\alpha|$.
When $m=1$, applying the same change of variables separately on the two
half-lines gives the following density.  Indeed, writing the integral on
$\mathbb{R}$ as the sum of its positive and negative half-line parts yields
$$
 \mu_{F,\alpha}(B)
 =\int_0^\infty\mathbf 1_B(a_\alpha+r^2)
 \bigl(|F_\alpha(r)|^2+|F_\alpha(-r)|^2\bigr)\,\mathrm dr.
$$
Since $\mathrm dr=\mathrm dt/(2\sqrt{t-a_\alpha})$, this becomes
$\mu_{F,\alpha}(B)=\int_B\rho_{F,\alpha}(t)\,\mathrm dt$, where
$$
 \rho_{F,\alpha}(t)
 =\frac{|F_\alpha(\sqrt{t-d-2|\alpha|})|^2
       +|F_\alpha(-\sqrt{t-d-2|\alpha|})|^2}
      {2\sqrt{t-d-2|\alpha|}}
 \mathbf 1_{(d+2|\alpha|,\infty)}(t).
$$
Thus the conclusion includes the one-dimensional free direction without any
implicit convention.
The possible threshold singularity in the last formula causes no difficulty:
by the preceding change of variables, its integral over
$(a_\alpha,a_\alpha+\varepsilon)$ equals
$$
 \int_{-\sqrt\varepsilon}^{\sqrt\varepsilon}|F_\alpha(\xi)|^2\,\mathrm d\xi,
$$
which is finite.  Hence, in every dimension $m\geq1$,
$\mu_{F,\alpha}\ll\mathrm dt$.

It remains to spell out why this proves that the spectrum is purely
absolutely continuous.  If $B$ has one-dimensional Lebesgue measure zero,
then $\mu_{F,\alpha}(B)=0$ for every $\alpha$, and therefore
$$
 \mu_F(B)=\sum_\alpha\mu_{F,\alpha}(B)=0.
$$
Thus $\mu_F\ll\mathrm dt$.  Since $F$ was arbitrary, every scalar spectral
measure of the direct-sum operator, and hence by unitary equivalence every
scalar spectral measure of $H$, is absolutely continuous.  Equivalently,
its absolutely continuous spectral subspace is the whole space:
$$
 \mathcal H_{\rm ac}(H)=L^2(\mathbb{R}^N).
$$
The canonical spectral decomposition
$$
 L^2(\mathbb{R}^N)
 =\mathcal H_{\rm ac}(H)\oplus
  \mathcal H_{\rm sc}(H)\oplus
  \mathcal H_{\rm pp}(H)
$$
then implies
$$
 \mathcal H_{\rm sc}(H)=\mathcal H_{\rm pp}(H)=\{0\}.
$$
Consequently,
$$
 \sigma_{\rm sc}(H)=\sigma_{\rm p}(H)=\varnothing,
 \qquad
 \sigma_{\rm ac}(H)=\sigma(H)=[d,\infty),
$$
which proves \eqref{eq:20260514-0958}.

Finally, if $u\in X=Q(H)$ and $F=\mathcal Uu$, then the quadratic form is
represented by
$$
 \mathfrak{h}[u,u]
 =\sum_\alpha\int_{\mathbb{R}^m}
 \bigl(a_\alpha+|\xi|^2\bigr)|F_\alpha(\xi)|^2\,\mathrm d\xi.
$$
Since $a_\alpha+|\xi|^2\geq d$ for every $\alpha$ and $\xi$, Parseval's
identity gives
$$
 \mathfrak{h}[u,u]\geq d\sum_\alpha\|F_\alpha\|_{L^2(\mathbb{R}^m)}^2
 =d\|u\|_{L^2(\mathbb{R}^N)}^2,
$$
which is precisely \eqref{eq4-1}.
\end{proof}

\begin{remark}\label{remark:20260518-1736}
For every $\varepsilon\in[0,1)$, scaling the $y$-variables gives
$$
 \sigma\bigl(-(1-\varepsilon)\Delta_y+|y|^2\bigr)
 =\left\{\sqrt{1-\varepsilon}\,(d+2k):k=0,1,2,\ldots\right\}.
$$
Combining this with the free Fourier variable yields
\begin{equation}\label{eq:deformed-bottom}
 \inf\sigma\bigl(-(1-\varepsilon)\Delta+|y|^2\bigr)
 =d\sqrt{1-\varepsilon}.
\end{equation}
\end{remark}

The next lemma is the rigorous ground-state projection identity.  It is deliberately formulated at the level of the quadratic form; no assumption $u\in D(H_y)$ is made.

\begin{lemma}\label{lemma:20260525-0923}
Let $u\in C^2(\mathbb{R}^N)$.  Assume either $u\in X$ or that $u$ satisfies \ref{Assumption:A}.  Then, for every $\eta\in C_c^\infty(\mathbb{R}^m)$,
\begin{equation}\label{eq:form-ground-pairing}
 \int_{\mathbb{R}^N}
 \left(\nabla_yu\cdot\nabla_y\phi_0+|y|^2u\phi_0-du\phi_0\right)
 \eta(z)\,\mathrm{d}x=0.
\end{equation}
Equivalently, for almost every $z$ in the $X$-case,
\begin{equation}\label{eq2-2}
 q_y\bigl(u(\cdot,z),\phi_0\bigr)
 =d\int_{\mathbb{R}^d}u(y,z)\phi_0(y)\,\mathrm{d}y,
\end{equation}
where $q_y$ is the closed quadratic form of $H_y$.
\end{lemma}

\begin{proof}
Suppose first that $u\in X$.  Fubini gives $u(\cdot,z)\in Q(H_y)$ for almost every $z$.  Since $C_c^\infty(\mathbb{R}^d)$ is dense in $Q(H_y)$ with respect to the form norm and $\phi_0\in D(H_y)$,
$$
 q_y(v,\phi_0)=\langle v,H_y\phi_0\rangle_{L^2_y}
 =d\langle v,\phi_0\rangle_{L^2_y},
 \qquad v\in Q(H_y).
$$
Applying this to $v=u(\cdot,z)$ and integrating against $\eta(z)$ proves the claim.

Under \ref{Assumption:A}, choose $\chi\in C_c^\infty(\mathbb{R}^d)$ with $\chi=1$ on $B_1$, put $\chi_L(y)=\chi(y/L)$, and integrate by parts in $y$ with the compactly supported function $\chi_L\phi_0$.  The error terms are supported where $L\lesssim|y|\lesssim2L$ and contain only $u$, $\nabla_yu$, $\phi_0$ and derivatives of $\chi_L$.  On the compact $z$-support of $\eta$, condition \ref{Assumption:A} provides constants $C_K$ and $\kappa_K<1/2$ independent of $z$.  Hence every error is bounded by a polynomial in $L$ times
$e^{-(1/2-\kappa_K)L^2}$ and tends to zero.  Dominated convergence then gives \eqref{eq:form-ground-pairing}.
\end{proof}

We now prove the necessary spectral condition.  The same argument also applies to a nonnegative supersolution of $(H+\lambda)u\geq0$.

\begin{proof}[Proof of Theorem \ref{thm:nec-liou}\,(i)]
Assume by contradiction that $\lambda<-d$.  Let $\psi_R$ be the positive first Dirichlet eigenfunction of $-\Delta_z$ on $B_R^m$, with eigenvalue $\mu_R=R^{-2}\mu_1$, normalized arbitrarily and extended by zero outside $B_R^m$.  Choose $R$ so large that
$$
 d+\lambda+\mu_R<0.
$$
The zero extension satisfies, in the sense of distributions,
\begin{equation}\label{eq:dirichlet-extension}
 -\Delta_z\psi_R\leq\mu_R\psi_R,
\end{equation}
because the outward normal derivative of the positive first eigenfunction is strictly negative on $\partial B_R^m$.

Use $\chi_L(y)\phi_0(y)\psi_R(z)$ as a test function, with a standard smooth approximation in the $z$-variable if desired, and then let $L\rightarrow\infty$.  Lemma \ref{lemma:20260525-0923} treats the $y$-part.  The $z$-part is estimated by \eqref{eq:dirichlet-extension}; the $X$-case follows by form approximation and the \ref{Assumption:A} case by the locally uniform Gaussian bound.  We obtain
$$
 0\leq\int_{\mathbb{R}^N}g(u)\phi_0\psi_R\,\mathrm{d}x
 \leq(d+\lambda+\mu_R)
 \int_{\mathbb{R}^N}u\phi_0\psi_R\,\mathrm{d}x<0,
$$
which is impossible because $u>0$, $\phi_0>0$, and $\psi_R>0$ in $B_R^m$.  Thus $\lambda\geq-d$.
\end{proof}

\section{Liouville theorems at the zero-mass threshold}\label{sec:Liouville-results}

At $\lambda=-d$ the global lower bound for $H-d$ is zero.  We first record two projection-based Liouville principles which depend on the number $m=N-d$ of free variables.  We then turn to the stronger finite-energy pure-power theorem, whose proof uses additional structure and reaches the full Sobolev critical exponent in dimension $N$.

\begin{lemma}\label{lem:weak-superharmonic}
Let $k\geq1$.  Suppose $v\geq0$, $v\in L^q(\mathbb{R}^k)$, and $-\Delta v\geq0$ in $\mathcal D'(\mathbb{R}^k)$.  If $q\geq1$ when $k\leq2$, or
$1\leq q\leq k/(k-2)$ when $k>2$, then $v=0$ almost everywhere.
\end{lemma}

\begin{proof}
Let $\rho_\varepsilon$ be a standard nonnegative mollifier and set $v_\varepsilon=v*\rho_\varepsilon$.  Then $v_\varepsilon\geq0$, $-\Delta v_\varepsilon\geq0$ pointwise, and $v_\varepsilon\in L^q$.  The classical $L^q$ Liouville theorem for nonnegative superharmonic functions (see, e.g., \cite[Lemma A.2]{ikoma2014compactness}) gives $v_\varepsilon\equiv0$.  Since $v_\varepsilon\rightarrow v$ in $L^q$, the conclusion follows.
\end{proof}

\begin{theorem}\label{thm:Lp-Liouville}
Let $u\in C^2(\mathbb{R}^N)\cap L^q(\mathbb{R}^N)$ be nonnegative and satisfy
$$
 (H-d)u\geq0.
$$
Assume either $u\in X$ or \ref{Assumption:A}.  If $q\geq1$ when $m\leq2$, or
$1\leq q\leq m/(m-2)$ when $m>2$, then $u\equiv0$.
\end{theorem}

\begin{proof}
Define
\begin{equation}\label{eq:projected-v}
 v(z):=\int_{\mathbb{R}^d}u(y,z)\phi_0(y)\,\mathrm{d}y.
\end{equation}
If $q>1$, H\"older's inequality gives
$$
 |v(z)|^q\leq
 \left(\int_{\mathbb{R}^d}|u(y,z)|^q\,\mathrm{d}y\right)
 \left(\int_{\mathbb{R}^d}\phi_0^{q/(q-1)}\,\mathrm{d}y\right)^{q-1};
$$
for $q=1$ use $\|\phi_0\|_\infty$.  Hence $v\in L^q(\mathbb{R}^m)$.

Let $0\leq\eta\in C_c^\infty(\mathbb{R}^m)$ and use the $y$-truncated test
$\chi_L(y)\phi_0(y)\eta(z)$ in the distributional inequality.  Passing to the limit $L\rightarrow\infty$ by Lemma \ref{lemma:20260525-0923} gives
$$
 \int_{\mathbb{R}^m}v(z)(-\Delta_z\eta(z))\,\mathrm{d}z\geq0.
$$
Thus $-\Delta_zv\geq0$ in $\mathcal D'(\mathbb{R}^m)$.  Lemma \ref{lem:weak-superharmonic} gives $v=0$.  Since $u\geq0$ and $\phi_0>0$, this forces $u=0$ almost everywhere, hence identically by continuity.
\end{proof}

\begin{theorem}\label{thm:our-Liouville}
Let $u\in C^2(\mathbb{R}^N)$ be nonnegative and satisfy
$$
 (H-d)u\geq u^p.
$$
Assume either $u\in X$ or \ref{Assumption:A}.  If $p\geq1$ when $m\leq2$, or
$1\leq p\leq m/(m-2)$ when $m>2$, then $u\equiv0$.
\end{theorem}

\begin{proof}
For $p=1$, let $\psi_R$ be the first Dirichlet eigenfunction used in the proof of Theorem \ref{thm:nec-liou}\,(i), with $\mu_R<1$.  Testing with $\chi_L\phi_0\psi_R$ and letting $L\rightarrow\infty$ gives
$$
 \int u\phi_0\psi_R\,\mathrm{d}x
 \leq \int (H-d)u\,\phi_0\psi_R\,\mathrm{d}x
 \leq \mu_R\int u\phi_0\psi_R\,\mathrm{d}x.
$$
Since the integral is nonnegative and $\mu_R<1$, it must vanish.  The
same argument applies to every translate $\psi_R(\cdot-z_0)$, where
$z_0\in\mathbb{R}^m$ is arbitrary.  It therefore shows that
$$
 \int_{\mathbb{R}^d\times B_R(z_0)}
 u(y,z)\phi_0(y)\psi_R(z-z_0)\,\mathrm dx=0
 \qquad\text{for every }z_0\in\mathbb{R}^m.
$$
Since the integrand is nonnegative and the weight is strictly positive in
the corresponding cylinder, these translated cylinders cover
$\mathbb{R}^N$, and hence $u\equiv0$.

Assume now $p>1$.  Choose $\zeta\in C_c^\infty(\mathbb{R}^m)$, $0\leq\zeta\leq1$, with $\zeta=1$ on $B_{1/2}$ and $\operatorname{supp}\zeta\subset B_1$, and set
$$
 \alpha:=\frac{2p}{p-1},
 \qquad
 \psi_R(z):=\zeta(z/R)^\alpha.
$$
Then
\begin{equation}\label{eq:Delta-psiR-new}
 |\Delta_z\psi_R|
 \leq CR^{-2}\psi_R^{1/p}\mathbf1_{\{R/2<|z|<R\}}.
\end{equation}
Testing the inequality with $\chi_L(y)\phi_0(y)\psi_R(z)$ and passing to $L\rightarrow\infty$ gives
\begin{equation}\label{eq:IR-start-new}
 I_R:=\int_{\mathbb{R}^N}u^p\phi_0\psi_R\,\mathrm{d}x
 \leq -\int_{\mathbb{R}^N}u\phi_0\Delta_z\psi_R\,\mathrm{d}x.
\end{equation}
The limiting step is justified by the form pairing in the $X$-case and by the locally uniform version of \ref{Assumption:A} in the alternative case.

Let $A_R:=\{z\in\mathbb{R}^m:R/2<|z|<R\}$.  From
\eqref{eq:Delta-psiR-new}--\eqref{eq:IR-start-new} and H\"older's
inequality on $\mathbb{R}^d\times A_R$, we obtain
\begin{align*}
 I_R
 &\leq CR^{-2}\int_{\mathbb{R}^d\times A_R}
       u\phi_0\psi_R^{1/p}\,\mathrm{d}x\\
 &=CR^{-2}\int_{\mathbb{R}^d\times A_R}
       (u^p\phi_0\psi_R)^{1/p}\phi_0^{(p-1)/p}\,\mathrm{d}x\\
 &\leq CR^{-2}I_R^{1/p}
       \left(\int_{\mathbb{R}^d\times A_R}
       \phi_0(y)\,\mathrm{d}x\right)^{(p-1)/p}\\
 &\leq C R^{-2+\frac{m(p-1)}p}I_R^{1/p},
\end{align*}
because $\phi_0\in L^1(\mathbb{R}^d)$ and $|A_R|\leq CR^m$.
If $I_R=0$, the estimate below is immediate.  Otherwise, division by
$I_R^{1/p}$ gives
\begin{equation}\label{eq:IR-bound-new}
 I_R\leq C R^{m-\frac{2p}{p-1}}.
\end{equation}
If $m\leq2$, the exponent is negative for every $p>1$.  If $m>2$ and
$p<m/(m-2)$, it is again negative.  Hence $I_R\rightarrow0$ and $u\equiv0$.

At the endpoint $p=m/(m-2)$, estimate \eqref{eq:IR-bound-new} is
uniform in $R$.  Since $\psi_R=1$ on $B_{R/2}$,
$$
 \int_{\mathbb{R}^d\times B_{R/2}}u^p\phi_0\,\mathrm{d}x
 \leq I_R\leq C,
$$
and monotone convergence yields
$u^p\phi_0\in L^1(\mathbb{R}^N)$.  Returning to the annular estimate
above and putting
$$
 J_R:=\int_{\mathbb{R}^d\times A_R}
 u^p\phi_0\psi_R\,\mathrm{d}x,
$$
we use $m(p-1)/p=2$ to obtain
$$
 I_R\leq C J_R^{1/p}\longrightarrow0,
$$
because $0\leq J_R\leq\int_{\mathbb{R}^d\times A_R}u^p\phi_0$
and the latter integral tends to zero.  Thus the endpoint also gives
$u\equiv0$.
\end{proof}

Theorem \ref{thm:nec-liou}\,(ii) follows from the preceding two theorems.  Moreover, if $m\leq4$ and
$0\leq u\in C^2\cap X$ solves $(H-d)u=g(u)$ with $g\geq0$, then $u\in L^2$ and Theorem \ref{thm:Lp-Liouville} applies with $q=2$.  This proves Theorem \ref{thm:nec-liou}\,(iii-1).

We now establish the new critical finite-energy result for the pure-power equation.  For $N=2$, every finite exponent is Sobolev-subcritical; the usual cutoff test first yields $u\in X$, local bootstrap gives a classical solution, and Theorem \ref{thm:Lp-Liouville} applies because $m=1$.  Hence only $N\geq3$ requires a new argument.

Set
\begin{equation}\label{crit:eq:phi0}
 \phi_0(y):=\pi^{-d/4}e^{-|y|^2/2},
 \qquad H_y\phi_0=d\phi_0,
 \qquad \|\phi_0\|_{L^2(\mathbb{R}^d)}=1.
\end{equation}
For $N\geq3$ we consider
\begin{equation}\label{crit:eq:model}
 -\Delta u+(|y|^2-d)u=u^p,
 \qquad 0\leq u\in H^1(\mathbb{R}^N),
\end{equation}
where a weak solution means
\begin{equation}\label{crit:eq:weak}
 \int\nabla u\cdot\nabla\varphi\,\mathrm{d}x
 +\int(|y|^2-d)u\varphi\,\mathrm{d}x
 =\int u^p\varphi\,\mathrm{d}x,
 \qquad \varphi\in C_c^\infty(\mathbb{R}^N).
\end{equation}

\begin{theorem}\label{crit:thm:main}
Let $N\geq3$ and $1\leq d\leq N-1$.  If
$$
 1\leq p\leq p_c:=\frac{N+2}{N-2},
$$
then \eqref{crit:eq:model} has only the trivial nonnegative $H^1(\mathbb{R}^N)$ weak solution.
\end{theorem}

\subsection{Regularity and a half-space oscillator gap}\label{crit:sec:prelim}

\subsubsection{Regularity, decay, and the graph domain of
\texorpdfstring{$H$}{H}}

\begin{lemma}\label{crit:lem:regularity}
Let $1<p\leq p_c$, and let $0\leq u\in H^1(\R^N)$ be a weak solution
of \eqref{crit:eq:model}.  Then either $u\equiv0$, or
\begin{equation}\label{crit:eq:regularity-summary}
 u>0,\qquad
 u\in X\cap L^\infty(\R^N)\cap C^\infty(\R^N),
 \qquad \lim_{|x|\rightarrow\infty}u(x)=0.
\end{equation}
Moreover, $u^p\in L^2(\R^N)$, $u\in D(H)$, and
\begin{equation}\label{crit:eq:graph-regularity}
 D^2u,\quad |y|^2u,\quad |y|\nabla u\in L^2(\R^N).
\end{equation}
In particular, for $i=1,\ldots,d$,
\begin{equation}\label{crit:eq:qv-form-domain}
 \partial_{y_i}u,\quad (\partial_{y_i}+y_i)u
 \in H^1(\R^N)\cap L^2(\R^N,|y|^2\,\mathrm{d} x).
\end{equation}
\end{lemma}

\begin{proof}
Because $p+1\leq2^*$, the Sobolev embedding gives
$u^{p+1}\in L^1(\R^N)$.  Let $\eta_R\in C_c^\infty(\R^N)$ satisfy
$0\leq\eta_R\leq1$, $\eta_R=1$ in $B_R$, $\eta_R=0$ outside
$B_{2R}$, and $|\nabla\eta_R|\leq C/R$.  The admissible test
$\eta_R^2u$ can be obtained from compactly supported smooth tests by
the usual local $H^1$ approximation.  Equation \eqref{crit:eq:weak} gives
\begin{align}
 &\int\eta_R^2\bigl(|\nabla u|^2+|y|^2u^2\bigr)\,\mathrm{d} x \notag\\
 &\qquad=d\int\eta_R^2u^2\,\mathrm{d} x
       +\int\eta_R^2u^{p+1}\,\mathrm{d} x
       -2\int\eta_Ru\nabla u\cdot\nabla\eta_R\,\mathrm{d} x.\label{crit:eq:cutoff-energy}
\end{align}
Young's inequality bounds the last term in absolute value by
$$
 \frac12\int\eta_R^2|\nabla u|^2\,\mathrm{d} x
 +2\int u^2|\nabla\eta_R|^2\,\mathrm{d} x.
$$
Thus the left-hand side of \eqref{crit:eq:cutoff-energy}, with one half of
the gradient term, is bounded independently of $R$.  Fatou's lemma
shows that $\int |y|^2u^2<\infty$, hence $u\in X$.  In fact, the
cutoff error satisfies
$$
 \left|2\int\eta_Ru\nabla u\cdot\nabla\eta_R\,\mathrm{d} x\right|
 \leq\frac{C}{R}
 \|u\|_{L^2(B_{2R}\setminus B_R)}
 \|\nabla u\|_{L^2(B_{2R}\setminus B_R)}\longrightarrow0.
$$
The other terms in \eqref{crit:eq:cutoff-energy} converge by dominated
convergence (and by the just-proved $|y|u\in L^2$).  Thus letting
$R\rightarrow\infty$ is equivalent to testing the equation with $u$ and yields
the Nehari identity
\begin{equation}\label{crit:eq:nehari-early}
 \int_{\R^N}\bigl(|\nabla u|^2+|y|^2u^2-du^2\bigr)\,\mathrm{d} x
 =\int_{\R^N}u^{p+1}\,\mathrm{d} x.
\end{equation}

Since $u\geq0$, equation \eqref{crit:eq:model} implies in distributions
\begin{equation}\label{crit:eq:subsolution}
 -\Delta u=(d-|y|^2)u+u^p\leq (d+u^{p-1})u.
\end{equation}
The coefficient in the last expression has precisely the integrability
needed for the Brezis--Kato argument.  Indeed, if $p=p_c$, then
$$
 u^{p-1}\in L^{N/2}(\mathbb{R}^N),
$$
because $(p-1)N/2=2^*$.  If $p<p_c$, then
$2^*/(p-1)>N/2$; equivalently, after splitting into $\{u>1\}$ and
$\{u\leq1\}$, one has
$$
 u^{p-1}\in L^{N/2}(\mathbb{R}^N)+L^\infty(\mathbb{R}^N).
$$
Consequently $d+u^{p-1}\in L^{N/2}+L^\infty$.  The Brezis--Kato
iteration \cite{BrezisKato1979}, in the critical-growth form used in
\cite[Theorem 2.3]{BL-CMP}, therefore yields
\begin{equation}\label{crit:eq:all-finite-Lq}
 u\in L^q(\mathbb{R}^N)\qquad\text{for every finite }q\geq2.
\end{equation}
The same conclusion holds for the subsolution \eqref{crit:eq:subsolution},
because the iteration tests only with nonnegative truncations and powers of
$u$, so its inequality direction is preserved.

Choose $r>\max\{N/2,2\}$.  By \eqref{crit:eq:all-finite-Lq},
$u\in L^{pr}$; interpolating with $u\in L^2$ also gives $u\in L^r$.
Hence
$$
 f:=du+u^p\in L^r(\mathbb{R}^N).
$$
The standard translation-invariant local estimate for nonnegative Poisson
subsolutions, applied to $-\Delta u\leq f$, gives
\begin{equation}\label{crit:eq:local-subsolution-estimate}
 \|u\|_{L^\infty(B_1(x_0))}
 \leq C\left(\|u\|_{L^2(B_2(x_0))}
       +\|f\|_{L^r(B_2(x_0))}\right),
 \qquad x_0\in\mathbb{R}^N,
\end{equation}
where $C$ is independent of $x_0$.  Since $u\in L^2$ and
$f\in L^r$, the right-hand side is uniformly bounded; moreover it tends
to zero as $|x_0|\rightarrow\infty$, because both corresponding global tail
norms tend to zero.  We have proved
\begin{equation}\label{crit:eq:Linfty-decay}
 u\in L^\infty(\R^N),\qquad u(x)\longrightarrow0
 \quad\text{as }|x|\rightarrow\infty.
\end{equation}

Let $K\Subset K'\Subset\R^N$.  On $K'$ the function
$(d-|y|^2)u+u^p$ is bounded.  Interior $W^{2,q}$ estimates for every
$q<\infty$, followed by the embedding $W^{2,q}\hookrightarrow
C^{1,\alpha}$ for $q>N$, give $u\in C^{1,\alpha}(K)$.  The right-hand
side is then locally H\"older continuous, so the Schauder estimate
gives $u\in C^{2,\alpha}(K)$.  If $u\not\equiv0$, the strong maximum
principle applied to
$-\Delta u+(|y|^2-d-u^{p-1})u=0$ gives $u>0$.  On each compact set the positive
function $u$ is bounded away from zero; since $s\mapsto s^p$ is smooth
on $(0,\infty)$, a further local bootstrap gives
$u\in C^\infty(\R^N)$.  These standard regularity and maximum
principle steps may also be found in
\cite[Chapters 3 and 7]{Chen-Li-Book} and
\cite[Chapters 8--9]{GilbargTrudinger2001}.

Boundedness and $u\in L^2$ imply
\begin{equation}\label{crit:eq:up-L2}
 \|u^p\|_2\leq\|u\|_\infty^{p-1}\|u\|_2<\infty.
\end{equation}
The closed quadratic form of $H$ has domain $X$.  Since
$$
 \mathfrak{h}[u,\varphi]
 :=\int(\nabla u\cdot\nabla\varphi+|y|^2u\varphi)\,\mathrm{d} x
 =\int(du+u^p)\varphi\,\mathrm{d} x
$$
for $\varphi\in C_c^\infty$, density in $X$ and the first
representation theorem show that $u\in D(H)$ and
$Hu=du+u^p\in L^2$.  We use here the standard form realization of
Schr\"odinger operators; see \cite[Chapter VIII]{Reed-Simon-1980}.

For $f\in C_c^\infty(\R^N)$, integration by parts gives the exact
oscillator graph identity
\begin{equation}\label{crit:eq:graph-identity}
 \|Hf\|_2^2
 =\|\Delta f\|_2^2+\||y|^2f\|_2^2
   +2\||y|\nabla f\|_2^2-2d\|f\|_2^2.
\end{equation}
Indeed,
\begin{align*}
 \operatorname{Re}\int(-\Delta f)|y|^2\overline f\,\mathrm{d} x
 &=\operatorname{Re}\int\nabla f\cdot
       \nabla(|y|^2\overline f)\,\mathrm{d} x\\
 &=\int|y|^2|\nabla f|^2\,\mathrm{d} x
   +\frac12\int\nabla(|y|^2)\cdot\nabla|f|^2\,\mathrm{d} x\\
 &=\int|y|^2|\nabla f|^2\,\mathrm{d} x
   -\frac12\int\Delta(|y|^2)|f|^2\,\mathrm{d} x\\
 &=\int |y|^2|\nabla f|^2\,\mathrm{d} x-d\int|f|^2\,\mathrm{d} x.
\end{align*}
Here $\Delta|y|^2=2d$.  Substitution into
$\|(-\Delta+|y|^2)f\|_2^2$ proves
\eqref{crit:eq:graph-identity}.
We now justify carefully the passage from test functions to $D(H)$.  Under
the unitary transform $\mathcal U$ in \eqref{eq:direct-sum-spectral}, convergence in the graph norm of $H$ is equivalent to convergence with respect to
$$
 \sum_\alpha\int_{\mathbb{R}^m}
 \bigl(1+q_\alpha(\xi)^2\bigr)|F_\alpha(\xi)|^2\,\mathrm d\xi,
 \qquad q_\alpha(\xi)=d+2|\alpha|+|\xi|^2.
$$
Truncating first the Hermite index and the Fourier variable, and then
approximating the finitely many remaining components in the corresponding
weighted $L^2$ spaces by functions in $C_c^\infty(\mathbb{R}^m)$, we obtain
$f_n\in\mathcal S(\mathbb{R}^N)$ such that
$$
 f_n\longrightarrow u,
 \qquad Hf_n\longrightarrow Hu
 \quad\text{in }L^2(\mathbb{R}^N).
$$
Identity \eqref{crit:eq:graph-identity} is valid for Schwartz functions.
Applying it to $f_n-f_k$ shows that $\Delta f_n$, $|y|^2f_n$, and
$|y|\nabla f_n$ are Cauchy in $L^2$.  Their limits agree, in the sense of
distributions, with $\Delta u$, $|y|^2u$, and $|y|\nabla u$, respectively.
Thus \eqref{crit:eq:graph-identity} extends to every $u\in D(H)$.
It follows that
$\Delta u$, $|y|^2u$, and $|y|\nabla u$ belong to $L^2$.
The Fourier identity
$\sum_{j,k}\|\partial_{jk}u\|_2^2=\|\Delta u\|_2^2$ then gives
$D^2u\in L^2$.  Hence $\partial_{y_i}u\in H^1$, while
$|y|\nabla u\in L^2$ gives
$|y|\partial_{y_i}u\in L^2$.  Furthermore,
$$
 \nabla(y_i u)=e_i u+y_i\nabla u\in L^2,
 \qquad |y|y_i u\in L^2
$$
by $|y|\nabla u,|y|^2u\in L^2$.  These elementary consequences prove
\eqref{crit:eq:qv-form-domain} for both $\partial_{y_i}u$ and
$(\partial_{y_i}+y_i)u$.
\end{proof}

\subsubsection{The missing coercivity at the zero-mass threshold}

For $\mu\leq0$, use the notation used in the moving-plane argument below:
\begin{equation}\label{crit:eq:halfspace}
 \Sigma_\mu:=\{x=(y,z)\in\R^d\times\R^m:y_1<\mu\},
 \qquad T_\mu:=\partial\Sigma_\mu.
\end{equation}

\begin{lemma}\label{crit:lem:halfspace-gap}
For $0\leq\varepsilon<1$, $\mu\leq0$, and
$\eta\in H_0^1(\Sigma_\mu)$ with $|y|\eta\in L^2(\Sigma_\mu)$,
\begin{equation}\label{crit:eq:halfspace-gap}
 \int_{\Sigma_\mu}
 \bigl((1-\varepsilon)|\nabla\eta|^2+|y|^2\eta^2\bigr)\,\mathrm{d} x
 \geq(d+2)\sqrt{1-\varepsilon}
       \int_{\Sigma_\mu}\eta^2\,\mathrm{d} x.
\end{equation}
\end{lemma}

\begin{proof}
Put $a=1-\varepsilon$.  The full-line one-dimensional oscillator
$-a\,\mathrm{d}^2/\mathrm{d}t^2+t^2$ has eigenvalues
$(2k+1)\sqrt a$, $k=0,1,\ldots$.  Indeed, the change
$t=a^{1/4}s$ transforms it into
$$
 \sqrt a\left(-\frac{\mathrm{d}^2}{\mathrm{d}s^2}+s^2\right).
$$
On $(-\infty,0)$ with a Dirichlet
condition at zero, its first eigenfunction is the restriction of the
first odd Hermite function and its first eigenvalue is $3\sqrt a$.
This follows either by odd reflection and the Hermite expansion, or by
the one-dimensional Sturm oscillation theorem.  Domain monotonicity
therefore gives, for every $h\in H_0^1(-\infty,\mu)$,
\begin{equation}\label{crit:eq:1d-gap}
 \int_{-\infty}^{\mu}\bigl(a|h'|^2+t^2h^2\bigr)\,\mathrm{d} t
 \geq3\sqrt a\int_{-\infty}^{\mu}h^2\,\mathrm{d} t,
 \qquad \mu\leq0.
\end{equation}
For each of the remaining coordinates $y_j$, $2\leq j\leq d$, the
full-line ground-state inequality is
$$
 \int_{\R}\bigl(a|h'|^2+t^2h^2\bigr)\,\mathrm{d} t
 \geq\sqrt a\int_{\R}h^2\,\mathrm{d} t.
$$
Apply \eqref{crit:eq:1d-gap} to the $y_1$ slices of
$\eta\in C_c^\infty(\Sigma_\mu)$, apply this full-line inequality to
each $y_j$ slice, sum over $j$, and retain the nonnegative quantity
$a\int|\nabla_z\eta|^2$.  Fubini's theorem gives
\begin{align*}
 \int_{\Sigma_\mu}
 \bigl(a|\nabla\eta|^2+|y|^2\eta^2\bigr)\,\mathrm{d} x
 &\geq\bigl(3+(d-1)\bigr)\sqrt a
       \int_{\Sigma_\mu}\eta^2\,\mathrm{d} x\\
 &=(d+2)\sqrt a\int_{\Sigma_\mu}\eta^2\,\mathrm{d} x.
\end{align*}
Density proves the same statement on the asserted form domain.  This
is \eqref{crit:eq:halfspace-gap}.  This is consistent with the Hermite--Fourier decomposition in Section \ref{sec:Spectral}; see also \cite[Chapter VIII]{Reed-Simon-1980}.
\end{proof}

\subsection{Moving planes at \texorpdfstring{$\lambda=-d$}{lambda=-d}}
\label{crit:sec:moving-planes}

The global bottom of the spectrum of $H-d$ is zero, which is why the
coercivity used for $\lambda>-d$ disappears at the endpoint.
Lemma \ref{crit:lem:halfspace-gap} supplies a gap of size two on every
half-space encountered before the plane reaches the origin.

\begin{proposition}\label{crit:prop:y-symmetry}
Let $1<p\leq p_c$, and let $u$ be a nonzero nonnegative weak solution
of \eqref{crit:eq:model}.  Then there is a function $U$ such that
\begin{equation}\label{crit:eq:radial-u}
 u(y,z)=U(|y|,z),
 \qquad \partial_rU(r,z)<0
 \quad(r>0,\ z\in\R^m).
\end{equation}
\end{proposition}

\begin{proof}
By Lemma \ref{crit:lem:regularity}, $u>0$, $u\in C^2\cap X$, and
$u(x)\rightarrow0$ as $|x|\rightarrow\infty$.  For $\mu\leq0$, let
\begin{equation}\label{crit:eq:reflection}
 x^\mu=(2\mu-y_1,y_2,\ldots,y_d,z),\qquad
 u_\mu(x)=u(x^\mu),\qquad w_\mu=u_\mu-u.
\end{equation}
For $x\in\Sigma_\mu$,
\begin{equation}\label{crit:eq:potential-reflection}
 |y|^2-|y^\mu|^2=4\mu(y_1-\mu)\geq0.
\end{equation}
Subtracting the equation for $u$ from the reflected equation and
using \eqref{crit:eq:potential-reflection}, we first have the exact identity
$$
 -\Delta w_\mu+(|y|^2-d)w_\mu
 =u_\mu^p-u^p+(|y|^2-|y^\mu|^2)u_\mu.
$$
The last term is nonnegative, while
$u_\mu^p-u^p=p\psi_\mu^{p-1}(u_\mu-u)$.  Hence
\begin{equation}\label{crit:eq:w-ineq}
 -\Delta w_\mu+
 \bigl(|y|^2-d-p\psi_\mu^{p-1}\bigr)w_\mu\geq0
 \quad\text{in }\Sigma_\mu,\qquad w_\mu=0\quad\text{on }T_\mu,
\end{equation}
where $\psi_\mu(x)$ lies between $u_\mu(x)$ and $u(x)$.  More
explicitly, when $u_\mu\ne u$ it is defined by
$$
 p\psi_\mu^{p-1}
 =\frac{u_\mu^p-u^p}{u_\mu-u};
$$
at equality it may be set equal to $u$.  Notice that on
$\{w_\mu<0\}$ one has
\begin{equation}\label{crit:eq:psi-upper}
 0<u_\mu<\psi_\mu<u.
\end{equation}
For fixed $\mu$, the negative part
$w_\mu^-:=\min\{w_\mu,0\}$ belongs to
$H_0^1(\Sigma_\mu)$ and $|y|w_\mu^-\in L^2$.  This follows from
$u\in X$ after the change of variables $x\mapsto x^\mu$, using
$|y|^2\leq2|y^\mu|^2+8\mu^2$ on the reflected half-space.

\smallskip
\noindent\emph{Step 1: starting the plane.}
Choose $R$ so large that
\begin{equation}\label{crit:eq:start-small}
 p\,u(x)^{p-1}<1\qquad (|x|>R).
\end{equation}
If $\mu<-R$, then $\Sigma_\mu\subset\{|x|>R\}$.  In the weak
supersolution inequality use the admissible nonnegative test
$-w_\mu^-$.  Multiplying the resulting inequality by $-1$ and using
\eqref{crit:eq:psi-upper}, \eqref{crit:eq:start-small}, and Lemma
\ref{crit:lem:halfspace-gap} with $\varepsilon=0$, we get
\begin{align*}
 0&\geq\int_{\{w_\mu<0\}}
 \left(|\nabla w_\mu^-|^2+
       (|y|^2-d-p\psi_\mu^{p-1})(w_\mu^-)^2\right)\,\mathrm{d} x\\
 &\geq\bigl((d+2)-d-1\bigr)\|w_\mu^-\|_2^2.
\end{align*}
Hence $w_\mu^-\equiv0$ and $w_\mu\geq0$ for all sufficiently
negative $\mu$.

\smallskip
\noindent\emph{Step 2: continuation to the origin.}
Set
\begin{equation}\label{crit:eq:mu0}
 \mu_0:=\sup\{\bar\mu\leq0:
 w_\mu\geq0\text{ in }\Sigma_\mu
 \text{ for every }\mu\leq\bar\mu\}.
\end{equation}
Continuity gives $w_{\mu_0}\geq0$.  Suppose for contradiction that
$\mu_0<0$.  The function $w_{\mu_0}$ is not identically zero.  Indeed,
otherwise $u_{\mu_0}=u$, and subtracting their exact equations would
give
$$
 (|y|^2-|y^{\mu_0}|^2)u=0\quad\text{in }\Sigma_{\mu_0},
$$
contrary to $u>0$ and \eqref{crit:eq:potential-reflection}.  The strong
maximum principle and Hopf boundary lemma therefore give
\begin{equation}\label{crit:eq:strict-w0}
 w_{\mu_0}>0\quad\text{in }\Sigma_{\mu_0},
 \qquad \partial_\nu w_{\mu_0}<0\quad\text{on }T_{\mu_0}.
\end{equation}
Their use is legitimate although the zero-order coefficient in
\eqref{crit:eq:w-ineq} has no global sign: on every bounded set it is
bounded, and replacing it by its positive part preserves the
supersolution inequality for $w_{\mu_0}\geq0$.  See
\cite[Theorem 7.3.3]{Chen-Li-Book}.

By the definition of $\mu_0$, there are arbitrarily small
$\delta>0$, with $\delta<|\mu_0|$, for which
\begin{equation}\label{crit:eq:negative-nonempty}
 \{w_{\mu_0+\delta}<0\}\cap\Sigma_{\mu_0+\delta}\ne\varnothing.
\end{equation}
Indeed, if this failed, there would exist
$\delta_0\in(0,|\mu_0|)$ such that
$w_{\mu_0+\delta}\geq0$ in $\Sigma_{\mu_0+\delta}$ for every
$0<\delta<\delta_0$.  Together with the defining property for all
planes to the left of $\mu_0$, this would put
$\mu_0+\delta_0/2$ in the set whose supremum is $\mu_0$, a
contradiction.  We show that \eqref{crit:eq:negative-nonempty} is
impossible.  Fix $R>0$ and put
\begin{equation}\label{crit:eq:OmegaRdelta}
 \Omega_{R,\delta}:=
 B_R\cap\Sigma_{\mu_0+\delta}\cap
 \{w_{\mu_0+\delta}<0\}.
\end{equation}
Since $u\in C^1$, the reflected functions converge uniformly on every
compact set:
\begin{equation}\label{crit:eq:compact-convergence}
 w_{\mu_0+\delta}\longrightarrow w_{\mu_0}
 \quad\text{uniformly on }B_R.
\end{equation}
Let
$a(\delta)=\|w_{\mu_0+\delta}-w_{\mu_0}\|_{L^\infty(B_R)}$.
On $\Omega_{R,\delta}\cap\Sigma_{\mu_0}$ we have
$0<w_{\mu_0}<a(\delta)$.  Hence
$$
 \Omega_{R,\delta}\cap\Sigma_{\mu_0}
 \subset\{x\in B_R\cap\Sigma_{\mu_0}:
           0<w_{\mu_0}(x)<a(\delta)\}.
$$
The measure of the set on the right tends to zero by
\eqref{crit:eq:strict-w0} and continuity from above.  The remaining strip
$B_R\cap(\Sigma_{\mu_0+\delta}\setminus\Sigma_{\mu_0})$ also has
measure tending to zero.  Therefore
\begin{equation}\label{crit:eq:Omega-small}
 |\Omega_{R,\delta}|\longrightarrow0
 \quad\text{as }\delta\downarrow0.
\end{equation}

Let $2^*=2N/(N-2)$.  On $\Omega_{R,\delta}$,
\eqref{crit:eq:psi-upper} and local boundedness of $u$ give
$p\psi_{\mu_0+\delta}^{p-1}\leq K_R$.  By H\"older's and Sobolev's
inequalities, after extending the negative part by zero,
\begin{align}
 &\int_{\Omega_{R,\delta}}
 p\psi_{\mu_0+\delta}^{p-1}
 (w_{\mu_0+\delta}^-)^2\,\mathrm{d} x \notag\\
 &\qquad\leq K_R|\Omega_{R,\delta}|^{2/N}
 \|w_{\mu_0+\delta}^-\|_{2^*}^2
 \leq K_RS_N|\Omega_{R,\delta}|^{2/N}
 \|\nabla w_{\mu_0+\delta}^-\|_2^2.\label{crit:eq:small-measure-estimate}
\end{align}
On the complement of $B_R$, the bound
\eqref{crit:eq:psi-upper} and $u(x)\rightarrow0$ give, uniformly in small $\delta$,
\begin{equation}\label{crit:eq:tail-estimate}
 \int_{B_R^c\cap\{w_{\mu_0+\delta}<0\}}
 p\psi_{\mu_0+\delta}^{p-1}(w_{\mu_0+\delta}^-)^2\,\mathrm{d} x
 \leq o_R(1)\|w_{\mu_0+\delta}^-\|_2^2.
\end{equation}
Consequently, given $0<\varepsilon<1$, first choosing $R$ large and
then $\delta$ small in \eqref{crit:eq:Omega-small} yields
\begin{equation}\label{crit:eq:nonlinear-error}
 \int_{\{w_{\mu_0+\delta}<0\}}
 p\psi_{\mu_0+\delta}^{p-1}(w_{\mu_0+\delta}^-)^2\,\mathrm{d} x
 \leq\varepsilon\left(
 \|\nabla w_{\mu_0+\delta}^-\|_2^2
 +\|w_{\mu_0+\delta}^-\|_2^2\right).
\end{equation}

Choose $\varepsilon>0$ so small that
\begin{equation}\label{crit:eq:kappa-positive}
 \kappa_\varepsilon:=(d+2)\sqrt{1-\varepsilon}-d-\varepsilon>0.
\end{equation}
Using the nonnegative test $-w_{\mu_0+\delta}^-$ in
\eqref{crit:eq:w-ineq}, multiplying by $-1$, invoking
\eqref{crit:eq:nonlinear-error}, and then applying Lemma
\ref{crit:lem:halfspace-gap} on $\Sigma_{\mu_0+\delta}$ (which is allowed
because $\mu_0+\delta<0$), we obtain
\begin{align*}
 0&\geq
 \int\left((1-\varepsilon)|\nabla w_{\mu_0+\delta}^-|^2
       +|y|^2(w_{\mu_0+\delta}^-)^2\right)\,\mathrm{d} x
 -(d+\varepsilon)\|w_{\mu_0+\delta}^-\|_2^2\\
 &\geq\kappa_\varepsilon
 \|w_{\mu_0+\delta}^-\|_2^2.
\end{align*}
Thus the negative part vanishes, contradicting
\eqref{crit:eq:negative-nonempty}.  Therefore $\mu_0=0$.

To run the argument from the opposite side without changing any sign,
apply the already proved left-half-space result to
$$
 \widetilde u(y_1,y_2,\ldots,y_d,z)
 :=u(-y_1,y_2,\ldots,y_d,z),
$$
which solves the same equation.  The two resulting inequalities at
the plane $y_1=0$ are opposite to one another, and therefore
$u(-y_1,y',z)=u(y_1,y',z)$.  More generally, for any unit vector
$e\in\R^d$, rotate the $y$ coordinates so that $e$ becomes the first
coordinate vector.  Since $|y|^2$ is invariant under that rotation,
the same proof shows that $u$ is invariant under reflection in
$\{y\cdot e=0\}$.  Such reflections generate $O(d)$; hence
$u(y,z)=U(|y|,z)$.  This is the classical moving-plane
conclusion of \cite{GidasNiNirenberg1979}, with the endpoint
coercivity supplied here by Lemma \ref{crit:lem:halfspace-gap}.

Finally, for every $\mu<0$, $w_\mu$ is not identically zero, so the
Hopf conclusion in \eqref{crit:eq:strict-w0} holds at $T_\mu$.  Since the
outward normal to $\Sigma_\mu$ is $e_1$ and
$$
 \partial_{y_1}w_\mu=-2\partial_{y_1}u
 \quad\text{on }T_\mu,
$$
we obtain $\partial_{y_1}u>0$ on $y_1=\mu<0$.  Evaluating at
$y=-re_1$ gives $\partial_rU(r,z)<0$ for every $r>0$.
\end{proof}

\subsection{An annihilation-operator inequality}\label{crit:sec:annihilation}

The symmetry of $u$ is not by itself sufficient for the final
Pohozaev contradiction.  We need the stronger fact that $u$ decreases
relative to the oscillator ground state.

Define the linearized operator
\begin{equation}\label{crit:eq:linearized}
 \mathcal L:=-\Delta+|y|^2-d-pu^{p-1}.
\end{equation}

\begin{proposition}\label{crit:prop:relative-monotonicity}
Under the assumptions of Proposition \ref{crit:prop:y-symmetry}, set
\begin{equation}\label{crit:eq:W-def}
 W(y,z):=\frac{u(y,z)}{\phi_0(y)}.
\end{equation}
Then $W(y,z)=\mathcal W(|y|,z)$ and
\begin{equation}\label{crit:eq:W-strict}
 \partial_r\mathcal W(r,z)<0
 \qquad(r>0,\ z\in\R^m).
\end{equation}
\end{proposition}

\begin{proof}
Fix $i\in\{1,\ldots,d\}$ and let
$\Omega_i:=\{(y,z):y_i>0\}$.  By Proposition
\ref{crit:prop:y-symmetry},
\begin{equation}\label{crit:eq:q-positive}
 q_i:=-\partial_{y_i}u>0\quad\text{in }\Omega_i,
 \qquad q_i=0\quad\text{on }\partial\Omega_i.
\end{equation}
Differentiating \eqref{crit:eq:model} in $y_i$ gives
\begin{equation}\label{crit:eq:q-equation}
 \mathcal L q_i=2y_i u>0\quad\text{in }\Omega_i.
\end{equation}
All differentiations and tests below are justified by
\eqref{crit:eq:qv-form-domain}; alternatively they may first be performed
with compact cutoffs and then passed to the limit.

For $\eta\in C_c^\infty(\Omega_i)$, the ground-state, or Picone,
identity gives
\begin{align}
 \mathcal Q_i[\eta]
 &:=\int_{\Omega_i}
 \left(|\nabla\eta|^2+
 (|y|^2-d-pu^{p-1})\eta^2\right)\,\mathrm{d} x \notag\\
 &=\int_{\Omega_i}q_i^2
   \left|\nabla\left(\frac{\eta}{q_i}\right)\right|^2\,\mathrm{d} x
 +\int_{\Omega_i}\frac{\mathcal L q_i}{q_i}\eta^2\,\mathrm{d} x
 \geq0.\label{crit:eq:picone}
\end{align}
Here is the full calculation behind \eqref{crit:eq:picone}.  Since $q_i>0$
on the compact support of $\eta$,
\begin{align*}
 q_i^2\left|\nabla\left(\frac{\eta}{q_i}\right)\right|^2
 &=|\nabla\eta|^2
   -2\frac{\eta}{q_i}\nabla\eta\cdot\nabla q_i
   +\frac{\eta^2}{q_i^2}|\nabla q_i|^2\\
 &=|\nabla\eta|^2
   -\nabla q_i\cdot\nabla\left(\frac{\eta^2}{q_i}\right).
\end{align*}
Integration by parts, with no boundary term because
$\eta\Subset\Omega_i$, gives
$$
 \int q_i^2\left|\nabla\left(\frac{\eta}{q_i}\right)\right|^2
 =\int|\nabla\eta|^2
  +\int\frac{\Delta q_i}{q_i}\eta^2.
$$
Adding the potential term and using
$$
 \frac{\mathcal L q_i}{q_i}
 =-\frac{\Delta q_i}{q_i}+|y|^2-d-pu^{p-1}
$$
proves \eqref{crit:eq:picone} term by term.  Thus no spectral assertion
about $\mathcal L$ is being assumed.  This identity is the linear
Allegretto--Piepenbrink ground-state transform; see
\cite{Allegretto1974,Piepenbrink1974}.  Since $u$ is bounded, the
negative part of the potential in \eqref{crit:eq:linearized} is bounded.
To see the required density, first truncate a form-domain function by
cutoffs in $B_R$, then translate the cutoff a distance $o(1)$ into
$\Omega_i$, and finally mollify; the $H^1$ and $|y|^2$-weighted errors
tend to zero in this order.  Hence $C_c^\infty(\Omega_i)$ is dense in
the space below, and $\mathcal Q_i\geq0$ extends to
\begin{equation}\label{crit:eq:form-domain-half}
 H_0^1(\Omega_i)\cap L^2(\Omega_i,|y|^2\,\mathrm{d} x).
\end{equation}

Introduce the oscillator annihilation derivative
\begin{equation}\label{crit:eq:v-def}
 v_i:=A_iu,\qquad \hbox{with}~A_i:=\partial_{y_i}+y_i.
\end{equation}
The commutator identity
\begin{equation}\label{crit:eq:commutator}
 [H-d,A_i]
 =-2A_i
\end{equation}
follows directly from
$$
 [-\partial_{y_i}^2,y_i]=-2\partial_{y_i},
 \qquad [y_i^2,\partial_{y_i}]=-2y_i,
$$
all other summands of $H-d$ commuting with
$A_i$.  Together with $(H-d)u=u^p$, it implies
\begin{equation}\label{crit:eq:v-equation}
 (\mathcal L+2)v_i=-(p-1)y_i u^p<0
 \quad\text{in }\Omega_i.
\end{equation}
For clarity, the nonlinear computation is
\begin{align*}
 (H-d)v_i=&(H-d)A_iu=\big(A_i(H-d)-2A_i\big)u\\
 &=A_i(u^p)-2v_i\\
 &=pu^{p-1}\partial_{y_i}u+y_i u^p-2v_i\\
 &=pu^{p-1}(A_iu-y_iu)+y_i u^p-2v_i\\
 &=\big(pu^{p-1}-2\big)v_i-(p-1)y_i u^p,
\end{align*}
which gives \eqref{crit:eq:v-equation}.

The $y_i$-evenness of $u$ gives
$\partial_{y_i}u\big|_{y_i=0}=0$ and hence
$v_i=A_i u=0$ on $\partial\Omega_i$.  Lemma
\ref{crit:lem:regularity} gives
$v_i\in H^1(\mathbb{R}^N)\cap L^2(\mathbb{R}^N,|y|^2\,\mathrm{d}x)$;
therefore its restriction, whose trace on $\partial\Omega_i$ vanishes,
belongs to
$$
 H_0^1(\Omega_i)\cap L^2(\Omega_i,|y|^2\,\mathrm{d}x).
$$
We may consequently test \eqref{crit:eq:v-equation} with
$v_i^+:=\max\{v_i,0\}$ and use \eqref{crit:eq:picone}:
\begin{equation}\label{crit:eq:v-positive-test}
 0\leq 2\|v_i^+\|_2^2\leq \mathcal Q_i[v_i^+]+2\|v_i^+\|_2^2
 =-(p-1)\int_{\Omega_i}y_i u^p v_i^+\,\mathrm{d} x\leq0.
\end{equation}
The right-hand side is finite because
$|y_i|u^p\leq\|u\|_\infty^{p-1}|y|u\in L^2$ and $v_i^+\in L^2$.
It follows that $v_i^+\equiv0$, so $v_i\leq0$.  In fact $v_i<0$ in
$\Omega_i$: the nonnegative function $-v_i$ satisfies
$$
 (\mathcal L+2)(-v_i)=(p-1)y_i u^p>0,
$$
and the strong maximum principle excludes an interior zero.

Finally, \eqref{crit:eq:phi0} gives the exact identity
\begin{equation}\label{crit:eq:v-W}
 v_i=(\partial_{y_i}+y_i)u
 =\phi_0\,\partial_{y_i}\left(\frac{u}{\phi_0}\right)
 =\phi_0\partial_{y_i}W.
\end{equation}
Thus $\partial_{y_i}W<0$ in $\Omega_i$.  Both $u$ and $\phi_0$ are
radial in $y$, so $W=\mathcal W(|y|,z)$, and taking $y=re_i$ proves
\eqref{crit:eq:W-strict}.
\end{proof}

\subsection{The strict Gaussian second-moment inequality}
\label{crit:sec:gaussian-moment}

\begin{proposition}\label{crit:prop:moment}
Let $u$ be as in Proposition \ref{crit:prop:relative-monotonicity}.  Then
\begin{equation}\label{crit:eq:moment-main}
 \int_{\R^N}|y|^2u^2\,\mathrm{d} x
 <\frac d2\int_{\R^N}u^2\,\mathrm{d} x.
\end{equation}
\end{proposition}

\begin{proof}
Let
\begin{equation}\label{crit:eq:gamma}
 \,\mathrm{d}\gamma(y):=\phi_0(y)^2\,\mathrm{d} y
 =\pi^{-d/2}e^{-|y|^2}\,\mathrm{d} y.
\end{equation}
This is a probability measure and
\begin{equation}\label{crit:eq:gamma-second}
 \int_{\R^d}|y|^2\,\mathrm{d}\gamma(y)=\frac d2.
\end{equation}
For almost every fixed $z$, put
$F_z(y)=W(y,z)^2$.  Proposition
\ref{crit:prop:relative-monotonicity} says that $F_z$ is a strictly
decreasing function of $|y|$.  The elementary covariance identity
\begin{align}
 &\int_{\R^d}\left(|y|^2-\frac d2\right)F_z(y)\,\mathrm{d}\gamma(y)
 \notag\\
 &\quad=\frac12\iint_{\R^d\times\R^d}
 \bigl(|y|^2-|y'|^2\bigr)
 \bigl(F_z(y)-F_z(y')\bigr)
 \,\mathrm{d}\gamma(y)\,\mathrm{d}\gamma(y')\label{crit:eq:covariance}
\end{align}
is obtained as follows.  If $g(y)=|y|^2$, then expansion of the four
terms in the product gives
\begin{align*}
 &\frac12\iint(g(y)-g(y'))(F_z(y)-F_z(y'))
     \,\mathrm{d}\gamma(y)\,\mathrm{d}\gamma(y')\\
 &=\int gF_z\,\mathrm{d}\gamma
   -\left(\int g\,\mathrm{d}\gamma\right)
    \left(\int F_z\,\mathrm{d}\gamma\right),
\end{align*}
which is the left-hand side of \eqref{crit:eq:covariance} by
\eqref{crit:eq:gamma-second}.  This use of Fubini is legitimate: the
absolute value of the double-integral integrand is bounded by
$(g(y)+g(y'))(F_z(y)+F_z(y'))$, whose integral is finite because both
$\int F_z\,\mathrm{d}\gamma$ and $\int gF_z\,\mathrm{d}\gamma$ are finite.  The
integrand on the right of \eqref{crit:eq:covariance} is nonpositive,
because one factor increases and the other strictly decreases with
the radius.  It is strictly negative on a set of positive
$\gamma\otimes\gamma$ measure: choose $0<r_1<r_2$ and small radial
intervals around $r_1,r_2$; strict continuity and monotonicity of
$F_z$ give opposite strict signs for the two differences on the
product of the corresponding annuli.  Therefore
\begin{equation}\label{crit:eq:slice-moment}
 \int_{\R^d}|y|^2F_z\,\mathrm{d}\gamma
 <\frac d2\int_{\R^d}F_z\,\mathrm{d}\gamma.
\end{equation}
All integrals are finite for almost every $z$, since
$u\in X$.  Indeed, $F_z\,\mathrm{d}\gamma=u(y,z)^2\,\mathrm{d} y$.
By \eqref{crit:eq:slice-moment},
$$D(z):=\int_{\mathbb{R}^d}\left(|y|^2-\frac{d}{2}\right)u(y,z)^2 \mathrm{d}y$$
is a strictly negative integrable function of $z$ and
$$|D(z)|\leq \int_{\mathbb{R}^d}\left(|y|^2+\frac{d}{2}\right)u(y,z)^2 \mathrm{d}y.$$
Then
$$\int_{\mathbb{R}^m}|D(z)|\mathrm{d}z\leq \int_{\mathbb{R}^N}\left(|y|^2+\frac{d}{2}\right)u(y,z)^2 \mathrm{d}y\mathrm{d}z<\infty.$$
Now, we collect that
$$D(z)<0~\hbox{for a.e.}~z, \qquad D\in L^1(\mathbb{R}^m),$$
and thus Fubini's theorem proves \eqref{crit:eq:moment-main}.
\end{proof}

\subsection{Anisotropic Pohozaev identities}\label{crit:sec:pohozaev}

Use the abbreviations
\begin{equation}\label{crit:eq:quantities}
 \begin{aligned}
 A_y&:=\int_{\R^N}|\nabla_yu|^2\,\mathrm{d} x,&
 A_z&:=\int_{\R^N}|\nabla_zu|^2\,\mathrm{d} x,\\
 B&:=\int_{\R^N}|y|^2u^2\,\mathrm{d} x,&
 M&:=\int_{\R^N}u^2\,\mathrm{d} x,&
 C&:=\int_{\R^N}u^{p+1}\,\mathrm{d} x.
 \end{aligned}
\end{equation}

\begin{lemma}
\label{crit:lem:pohozaev}
Let $1<p\leq p_c$ and let $u$ solve \eqref{crit:eq:model}.  Then
\begin{align}
 A_y+A_z+B-dM&=C,\label{crit:eq:nehari}\\
 2A_z&=\frac{m(p-1)}{p+1}C,\label{crit:eq:z-pohozaev}\\
 2(A_y-B)&=\frac{d(p-1)}{p+1}C.\label{crit:eq:y-pohozaev}
\end{align}
Consequently,
\begin{equation}\label{crit:eq:key-identity}
 2B-dM
 =\frac{N+2-(N-2)p}{m(p-1)}A_z.
\end{equation}
\end{lemma}

\begin{proof}
Identity \eqref{crit:eq:nehari} is \eqref{crit:eq:nehari-early}.  We give a
cutoff proof of the two Pohozaev identities so that no differentiability
of an unbounded dilation orbit is assumed.

Put $V_0(y)=|y|^2-d$.  If
$X\in C_c^1(\R^N;\R^N)$, multiplication of
$-\Delta u+V_0u=u^p$ by $X\cdot\nabla u$ and integration by parts
give the variational identity
\begin{align}
 0=\int_{\R^N}\biggl\{&
 (DX\nabla u)\cdot\nabla u
 -\frac12(\operatorname{div}X)|\nabla u|^2\notag\\
 &-\frac12\bigl(V_0\operatorname{div}X
                 +X\cdot\nabla V_0\bigr)u^2
 +\frac{\operatorname{div}X}{p+1}u^{p+1}
 \biggr\}\,\mathrm{d} x.\label{crit:eq:variational-identity}
\end{align}
This is the standard Pohozaev--Pucci--Serrin identity
\cite{Pohozaev1965,PucciSerrin1986}.  In the present setting it follows
directly from the displayed integration by parts because
$u\in H^2_{\mathrm{loc}}$.  More explicitly, with repeated indices
summed,
\begin{align*}
 \int(-\Delta u)(X\cdot\nabla u)
 &=\int \partial_j u\,\partial_j(X_k\partial_k u)\\
 &=\int (\partial_jX_k)\partial_j u\partial_k u
   +\frac12\int X\cdot\nabla(|\nabla u|^2)\\
 &=\int(DX\nabla u)\cdot\nabla u
   -\frac12\int(\operatorname{div}X)|\nabla u|^2,
\end{align*}
whereas
\begin{align*}
 \int V_0u(X\cdot\nabla u)
 &=-\frac12\int\bigl(V_0\operatorname{div}X
          +X\cdot\nabla V_0\bigr)u^2,\\
 \int u^p(X\cdot\nabla u)
 &=-\frac1{p+1}\int(\operatorname{div}X)u^{p+1}.
\end{align*}
Moving the last expression to the left proves
\eqref{crit:eq:variational-identity} with every sign displayed.

To justify the unbounded vector fields used below, choose
$\chi\in C_c^\infty(\R^N)$ with $\chi=1$ on $B_1$, $\chi=0$
outside $B_2$, and set $\chi_R(x)=\chi(x/R)$.  Let $P_y$ and $P_z$
be the orthogonal projections onto the $y$ and $z$ coordinates.  Use
in \eqref{crit:eq:variational-identity}
$$
 X_R^y(x)=\chi_R(x)P_yx,
 \qquad X_R^z(x)=\chi_R(x)P_zx.
$$
The matrices $DX_R^y,DX_R^z$ and their divergences are uniformly
bounded and converge pointwise to $P_y,P_z$ and $d,m$, respectively.
Every term produced by differentiating $\chi_R$ is supported in
$B_{2R}\setminus B_R$ and is bounded by a constant times
\begin{equation}\label{crit:eq:integrable-density}
 |\nabla u|^2+(|y|^2+1)u^2+u^{p+1},
\end{equation}
which is integrable.  Thus dominated convergence, or equivalently a
tail estimate on the annuli, permits $R\rightarrow\infty$.

For $X=P_zx=(0,z)$, one has
$DX=P_z$, $\operatorname{div}X=m$, and
$X\cdot\nabla V_0=0$.  The limit of
\eqref{crit:eq:variational-identity} is
\begin{equation}\label{crit:eq:z-raw}
 A_z-\frac m2(A_y+A_z+B-dM)+\frac{m}{p+1}C=0.
\end{equation}
Substitution of \eqref{crit:eq:nehari} gives, without suppressing the
algebra,
$$
 A_z-\frac m2C+\frac{m}{p+1}C=0,
 \qquad
 A_z=\frac{m(p-1)}{2(p+1)}C,
$$
which is \eqref{crit:eq:z-pohozaev}.

For $X=P_yx=(y,0)$, one has
$DX=P_y$, $\operatorname{div}X=d$, and
$X\cdot\nabla V_0=2|y|^2$.  Hence
\begin{equation}\label{crit:eq:y-raw}
 A_y-\frac d2(A_y+A_z+B-dM)-B+\frac{d}{p+1}C=0.
\end{equation}
Using \eqref{crit:eq:nehari} yields
$$
 A_y-B-\frac d2C+\frac{d}{p+1}C=0,
 \qquad
 A_y-B=\frac{d(p-1)}{2(p+1)}C,
$$
which is \eqref{crit:eq:y-pohozaev}.

Multiplying \eqref{crit:eq:y-pohozaev} by $m$, multiplying
\eqref{crit:eq:z-pohozaev} by $d$, and subtracting eliminates $C$ and
gives
\begin{equation}\label{crit:eq:AyB}
 A_y-B=\frac dm A_z.
\end{equation}
Thus $A_y=B+(d/m)A_z$, and \eqref{crit:eq:nehari} becomes
$$
 C=\left(B+\frac dmA_z\right)+A_z+B-dM
  =2B+\frac NmA_z-dM.
$$
Rearrange this equality and use \eqref{crit:eq:z-pohozaev}:
\begin{align*}
 2B-dM
 &=C-\frac Nm A_z=\left(\frac{2(p+1)}{m(p-1)}-\frac Nm\right)A_z\\
 &=\frac{N+2-(N-2)p}{m(p-1)}A_z.
\end{align*}
This is \eqref{crit:eq:key-identity}.
\end{proof}

\subsection{Proof of the theorem}\label{crit:sec:proof-main}

\begin{proof}[Proof of Theorem \ref{crit:thm:main}]
First suppose $1<p\leq p_c$ and, seeking a contradiction, let
$u\not\equiv0$.  Propositions \ref{crit:prop:y-symmetry} and
\ref{crit:prop:relative-monotonicity} apply.  Proposition
\ref{crit:prop:moment} gives
\begin{equation}\label{crit:eq:negative-side}
 2B-dM<0.
\end{equation}
On the other hand, Lemma \ref{crit:lem:pohozaev} gives
\begin{equation}\label{crit:eq:nonnegative-side}
 2B-dM
 =\frac{N+2-(N-2)p}{m(p-1)}A_z\geq0,
\end{equation}
because $p\leq (N+2)/(N-2)$.  Equations
\eqref{crit:eq:negative-side} and \eqref{crit:eq:nonnegative-side} contradict
each other.  Hence $u\equiv0$.

It remains to treat $p=1$.  The equation is
$(H-d)u=u$.  Testing with $\eta_R^2u$ as in Lemma
\ref{crit:lem:regularity} gives explicitly
$$
 \int\eta_R^2(|\nabla u|^2+|y|^2u^2)\,\mathrm{d} x
 =(d+1)\int\eta_R^2u^2\,\mathrm{d} x
 -2\int\eta_Ru\nabla u\cdot\nabla\eta_R\,\mathrm{d} x.
$$
Young's inequality gives a bound independent of $R$, so Fatou's lemma
implies $u\in X$.  The form identity
$\mathfrak{h}[u,\varphi]=(d+1)\langle u,\varphi\rangle_{L^2}$ then
gives $u\in D(H)$.  Let
$\{\phi_\alpha\}_{\alpha\in\N_0^d}$ be the Hermite basis of
$L^2(\R^d)$, with
$$
 H_y\phi_\alpha=(d+2|\alpha|)\phi_\alpha.
$$
Set
$$
 u_\alpha(z):=\int_{\R^d}u(y,z)\phi_\alpha(y)\,\mathrm{d} y
 \in L^2(\R^m).
$$
Projecting $(H-d)u=u$ onto $\phi_\alpha$ gives
$(-\Delta_z+2|\alpha|)u_\alpha=u_\alpha$ in distributions.  Taking
the Fourier transform in $z$ gives, for almost every $\xi\in\R^m$,
\begin{equation}\label{crit:eq:p1-fourier}
 \bigl(2|\alpha|+|\xi|^2-1\bigr)
 \widehat u_\alpha(\xi)=0.
\end{equation}
For $|\alpha|\geq1$ the coefficient never vanishes.  For
$\alpha=0$, it vanishes only on the sphere $|\xi|=1$, a set of
$m$-dimensional Lebesgue measure zero.  An $L^2$ function supported on
that sphere is zero.  Thus every coefficient
$\widehat u_\alpha$ vanishes, and $u\equiv0$.  This is also consistent with Proposition \ref{pro:20260514-1010}.
\end{proof}

\subsection{A complementary necessary upper bound}
\label{crit:sec:upper-bound}

The following observation is not needed for Theorem \ref{crit:thm:main}, but it gives a complementary upper obstruction for any possible finite-energy pure-power solution.

\begin{proposition}
\label{crit:prop:upper-obstruction}
Let $p>1$, and suppose that
$0\leq u\in C^2(\R^N)\cap X\cap L^{p+1}(\R^N)$ is a nonzero
solution of \eqref{crit:eq:model}.  Then necessarily
\begin{equation}\label{crit:eq:pm-bound}
 p< p_m:=\frac{m+2}{m-2}
 \qquad\text{when }m>2.
\end{equation}
\end{proposition}

\begin{proof}
The proofs of \eqref{crit:eq:nehari} and \eqref{crit:eq:z-pohozaev} in Lemma
\ref{crit:lem:pohozaev} use the restriction $p\leq p_c$ only to obtain the
integrability and regularity established earlier.  Under the explicit
assumptions of the present proposition, the same cutoff calculations
apply verbatim because the density \eqref{crit:eq:integrable-density} is in
$L^1(\R^N)$.  We may therefore use both identities here.

The oscillator spectral inequality gives
\begin{equation}\label{crit:eq:Qy-positive}
 Q_y:=A_y+B-dM
 =\int_{\R^m}\langle(H_y-d)u(\cdot,z),u(\cdot,z)\rangle_{L^2_y}
 \,\mathrm{d} z\geq0.
\end{equation}
Since $u\not\equiv0$ and $u\geq0$, one has $C>0$.  By
\eqref{crit:eq:nehari}, $C=A_z+Q_y\geq A_z$.  By
\eqref{crit:eq:z-pohozaev},
\begin{equation}\label{crit:eq:Az-kC}
 A_z=\frac{m(p-1)}{2(p+1)}C.
\end{equation}
If $p>(m+2)/(m-2)$, the coefficient in
\eqref{crit:eq:Az-kC} is greater than one, contradicting $A_z\leq C$.

At $p=(m+2)/(m-2)$, equality must hold throughout, so $Q_y=0$.
For almost every $z$, set
$$
 q(z):=\int_{\mathbb{R}^d}
 \bigl(|\nabla_yu(y,z)|^2+(|y|^2-d)u(y,z)^2\bigr)\,\mathrm{d}y.
$$
The oscillator ground-state inequality gives $q(z)\geq0$ for almost every
$z$, while Fubini's theorem and \eqref{crit:eq:Qy-positive} give
$\int_{\mathbb{R}^m}q(z)\,\mathrm{d}z=Q_y=0$.  Hence $q(z)=0$ for
almost every $z$.  Since
$$
 \ker(H_y-d)=\operatorname{span}\{\phi_0\},
$$
equality in the form inequality on each such slice yields
\begin{equation}\label{crit:eq:ground-factor}
 u(y,z)=\phi_0(y)v(z)
\end{equation}
for almost every $(y,z)$, where
$v(z)=\langle u(\cdot,z),\phi_0\rangle_{L^2_y}\geq0$.  In fact
$v\in H^1(\mathbb{R}^m)$, as follows by taking the $z$-derivatives in the
factorization.  This factorization is incompatible with $(H-d)u=u^p$ when $p>1$.  To see this without a pointwise division,
set
$$
 c_p:=\int_{\R^d}\phi_0^{p+1}\,\mathrm{d} y,
 \qquad \psi:=\phi_0^p-c_p\phi_0.
$$
Then $\langle\psi,\phi_0\rangle=0$, while strict Cauchy--Schwarz gives
\begin{equation}\label{crit:eq:strict-CS}
 \langle\psi,\phi_0^p\rangle
 =\int_{\R^d}\phi_0^{2p}\,\mathrm{d} y-c_p^2>0,
\end{equation}
because $\phi_0^p$ is not proportional to $\phi_0$ for $p>1$.
Here $c_p^2$ in \eqref{crit:eq:strict-CS} is
$|\langle\phi_0^p,\phi_0\rangle|^2/
\|\phi_0\|_2^2$, because $\|\phi_0\|_2=1$; hence the strict inequality
is exactly the strict Cauchy--Schwarz inequality.

Choose $0\leq\eta\in C_c^\infty(\R^m)$ such that
$\int v^p\eta\,\mathrm{d} z>0$.
The test $\psi(y)\eta(z)$ is admissible by first replacing $\psi$ with
$\chi_L(y)\psi(y)$ and then letting $L\rightarrow\infty$: the linear terms
converge in the form norm because $\psi$ is a Schwartz function, while the
nonlinear term converges by H\"older's inequality, using
$u\in L^{p+1}(\mathbb{R}^N)$ and
$\psi(y)\eta(z)\in L^{p+1}(\mathbb{R}^N)$.
Testing the equation for \eqref{crit:eq:ground-factor} against
$\psi(y)\eta(z)$ gives on the left
\begin{align*}
 \int_{\R^m}v\eta\,\mathrm{d} z\,
   \langle(H_y-d)\phi_0,\psi\rangle_{L^2_y}
 +\langle\phi_0,\psi\rangle_{L^2_y}
   \langle-\Delta_zv,\eta\rangle=0,
\end{align*}
where both terms vanish by $(H_y-d)\phi_0=0$ and
$\langle\phi_0,\psi\rangle=0$.  The right-hand side equals
$$
 \langle\phi_0^p,\psi\rangle_{L^2_y}
 \int_{\R^m}v^p\eta\,\mathrm{d} z>0
$$
by \eqref{crit:eq:strict-CS}.
This contradiction excludes equality and proves \eqref{crit:eq:pm-bound}.
\end{proof}

The critical theorem above proves Theorem \ref{thm:nec-liou}\,(iii-2).  The complementary upper obstruction proves Theorem \ref{thm:nec-liou}\,(iii-3).  Combining them with Theorem \ref{thm:our-Liouville}, when $m\geq5$ a nontrivial finite-energy pure-power solution can therefore occur, at most, in the interval
$$
 \max\left\{\frac{m}{m-2},\frac{N+2}{N-2}\right\}
 <p<\frac{m+2}{m-2}.
$$
This identifies precisely what remains open in the sharpness problem; no existence in this interval is asserted here.

\section{\texorpdfstring{Existence and regularity for $\lambda>-d$}{Existence and regularity for lambda>-d}}\label{sec:existence}

Throughout this section $\lambda>-d$ and assumptions \ref{gcon-1}--\ref{gcon-2} are in force.  We use the classical mountain-pass theorem of Ambrosetti and Rabinowitz \cite{AmbrosettiRabinowitz1973}, together with a version of the concentration--compactness nonvanishing mechanism adapted to translations only in the free variables.  By Proposition \ref{pro:20260514-1010},
$$
 \|u\|_\lambda^2
 :=\int_{\mathbb{R}^N}\bigl(|\nabla u|^2+|y|^2u^2+\lambda u^2\bigr)\,\mathrm{d}x
$$
defines a norm equivalent to the standard form norm on $X$.  In particular,
\begin{equation}\label{eq:X-H1-embedding}
 \|u\|_{H^1(\mathbb{R}^N)}\leq C_\lambda\|u\|_\lambda,
 \qquad u\in X,
\end{equation}
and hence $X\hookrightarrow L^r(\mathbb{R}^N)$ for $2\leq r\leq2^*$ when $N\geq3$, and for every finite $r\geq2$ when $N=2$.

Extend $g$ by zero on $(-\infty,0]$.  The associated functional
\begin{equation}\label{eq:energy-functional}
 I(u):=\frac12\|u\|_\lambda^2-\int_{\mathbb{R}^N}G(u^+)\,\mathrm{d}x
\end{equation}
belongs to $C^1(X,\mathbb{R})$.  From \ref{gcon-1}, for every $\varepsilon>0$ there exist $C_\varepsilon>0$ and the exponent $q\in(2,2^*)$ from \ref{gcon-1} such that
\begin{equation}\label{eq:g-growth-existence}
 |g(s)|\leq\varepsilon |s|+C_\varepsilon |s|^{q-1},
 \qquad
 0\leq G(s)\leq\varepsilon s^2+C_\varepsilon |s|^q,
 \qquad s\geq0.
\end{equation}

\begin{lemma}\label{lemma:20260521-0952}
There exist $\rho,a>0$ such that $I(u)\geq a$ whenever $\|u\|_\lambda=\rho$, and there exists $e\in X$ with $\|e\|_\lambda>\rho$ and $I(e)<0$.
\end{lemma}

\begin{proof}
The first assertion follows from \eqref{eq:g-growth-existence} and the Sobolev embedding:
$$
 I(u)\geq\left(\frac12-C\varepsilon\right)\|u\|_\lambda^2-C_\varepsilon\|u\|_\lambda^q.
$$
Choose first $\varepsilon$ and then $\rho$ small.

Since $g\not\equiv0$ and $g\geq0$, there is $s_0>0$ with $G(s_0)>0$.  Assumption \ref{gcon-2} implies that $s\mapsto G(s)/s^\theta$ is nondecreasing on $(0,\infty)$ wherever $G>0$, so $G(s)\geq Cs^\theta$ for $s\geq s_0$.  Fix a nonzero $0\leq\eta\in C_c^\infty(\mathbb{R}^N)$.  Then $I(t\eta)\rightarrow-\infty$ as $t\rightarrow\infty$, and a sufficiently large multiple gives the required $e$.
\end{proof}

The only lack of compactness comes from translations in the $z$-variables.  We record the precise nonvanishing statement needed below. It is an anisotropic form of the vanishing principle underlying Lions's concentration--compactness method \cite{Lions1984}; a proof is included because only the free translations are relevant here.

\begin{lemma}\label{lem:anisotropic-lions}
Let $\{v_n\}$ be bounded in $X$.  Let $\{Q_k\}_{k\in\mathbb Z^m}$ be the unit cubes in $\mathbb{R}^m$.  If
\begin{equation}\label{eq:vanishing-cubes}
 \sup_{k\in\mathbb Z^m}
 \int_{\mathbb{R}^d\times Q_k}|v_n|^2\,\mathrm{d}x\longrightarrow0,
\end{equation}
then
\begin{equation}\label{eq:vanishing-Lr}
 v_n\longrightarrow0\quad\text{in }L^r(\mathbb{R}^N)
 \qquad\text{for every }2<r<2^*.
\end{equation}
Here $2^*=\infty$ when $N=2$.
\end{lemma}

\begin{proof}
Set $r_0:=2+4/N$.  Let $Q_k^*$ be a fixed enlargement of $Q_k$, chosen so that the family $\{Q_k^*\}$ has uniformly finite overlap.  By extending functions from $Q_k^*$ in the $z$-variables with an extension operator whose norm is independent of $k$, and then applying the standard Gagliardo--Nirenberg inequality in $\mathbb{R}^N$, one obtains
\begin{equation}\label{eq:local-GN-cubes}
 \|v\|_{L^{r_0}(\mathbb{R}^d\times Q_k)}^{r_0}
 \leq C
 \|v\|_{L^2(\mathbb{R}^d\times Q_k^*)}^{4/N}
 \|v\|_{H^1(\mathbb{R}^d\times Q_k^*)}^{2}.
\end{equation}
The constant is independent of $k$.  Summing in $k$ and using finite overlap gives
\begin{equation}\label{eq:global-cube-GN}
 \|v\|_{r_0}^{r_0}
 \leq C
 \left(\sup_k\|v\|_{L^2(\mathbb{R}^d\times Q_k)}\right)^{4/N}
 \|v\|_{H^1}^2,
\end{equation}
where replacing $Q_k^*$ by finitely many neighboring unit cubes only changes $C$.  Boundedness in $X$ and \eqref{eq:X-H1-embedding} therefore imply $v_n\rightarrow0$ in $L^{r_0}$.  Interpolation with the bounded Sobolev norms gives \eqref{eq:vanishing-Lr}.  In dimension $N=2$, one first obtains convergence in $L^4$ and then interpolates with any fixed finite Sobolev exponent larger than the desired $r$.
\end{proof}

\begin{proof}[Proof of Theorem \ref{thm:existence}]
Let
$$
 c:=\inf_{\gamma\in\Gamma}\max_{t\in[0,1]}I(\gamma(t)),
 \qquad
 \Gamma:=\{\gamma\in C([0,1],X):\gamma(0)=0,\ \gamma(1)=e\}.
$$
Lemma \ref{lemma:20260521-0952} gives $c>0$.  The mountain-pass theorem yields a Palais--Smale sequence $\{u_n\}\subset X$ such that
\begin{equation}\label{eq:PS-sequence}
 I(u_n)\rightarrow c,
 \qquad
 I'(u_n)\rightarrow0\quad\text{in }X'.
\end{equation}
Testing $I'(u_n)$ against $u_n^-=\min\{u_n,0\}$ gives
$$
 \|u_n^-\|_\lambda^2=o(1)\|u_n^-\|_\lambda,
$$
because $g(u_n^+)u_n^-=0$ almost everywhere.  Thus $u_n^-\rightarrow0$ in $X$.
Since $u_n^+u_n^-=0$ and
$\nabla u_n^+\cdot\nabla u_n^-=0$ almost everywhere,
$$
 \|u_n\|_\lambda^2
 =\|u_n^+\|_\lambda^2+\|u_n^-\|_\lambda^2,
$$
whereas the nonlinear part of $I$ is already evaluated at $u_n^+$.
Consequently,
$$
 I(u_n^+)=I(u_n)-\frac12\|u_n^-\|_\lambda^2=I(u_n)+o(1).
$$
Moreover, writing $\langle\cdot,\cdot\rangle_\lambda$ for the
bilinear form associated with $\|\cdot\|_\lambda$, for every $v\in X$,
$$
 |\langle I'(u_n^+)-I'(u_n),v\rangle|
 =|\langle u_n^+-u_n,v\rangle_\lambda|
 \leq\|u_n^-\|_\lambda\|v\|_\lambda,
$$
so $I'(u_n^+)-I'(u_n)\rightarrow0$ in $X'$.  We may therefore replace the
Palais--Smale sequence by its positive part and assume $u_n\geq0$.

The Ambrosetti--Rabinowitz inequality gives
\begin{align*}
 c+o(1)+o(1)\|u_n\|_\lambda
 &=I(u_n)-\frac1\theta\langle I'(u_n),u_n\rangle\\
 &\geq\left(\frac12-\frac1\theta\right)\|u_n\|_\lambda^2,
\end{align*}
so $\{u_n\}$ is bounded in $X$.

We next prove nonvanishing.  If \eqref{eq:vanishing-cubes} held for $u_n$, Lemma \ref{lem:anisotropic-lions} would imply $u_n\rightarrow0$ in $L^r$ for every $2<r<2^*$.
More precisely, \eqref{eq:g-growth-existence} and boundedness in $X$ give,
for every $\varepsilon>0$,
$$
 \limsup_{n\rightarrow\infty}\int g(u_n)u_n\,\mathrm d x
 \leq C\varepsilon,
 \qquad
 \limsup_{n\rightarrow\infty}\int G(u_n)\,\mathrm d x
 \leq C\varepsilon,
$$
because $u_n\rightarrow0$ in $L^q$.  Letting $\varepsilon\downarrow0$ yields
$$
 \int g(u_n)u_n\,\mathrm{d}x=o(1),
 \qquad
 \int G(u_n)\,\mathrm{d}x=o(1).
$$
The identity $\langle I'(u_n),u_n\rangle=o(1)$ would force $\|u_n\|_\lambda\rightarrow0$, and hence $I(u_n)\rightarrow0$, contradicting $c>0$.  Therefore there exist $\delta>0$ and $k_n\in\mathbb Z^m$ such that
\begin{equation}\label{eq:local-mass-positive}
 \int_{\mathbb{R}^d\times Q_{k_n}}u_n^2\,\mathrm{d}x\geq\delta.
\end{equation}
Let $z_n$ be the center of $Q_{k_n}$ and set
$$
 \widetilde u_n(y,z):=u_n(y,z+z_n).
$$
The functional and the $X$-norm are invariant under this translation.  Passing to a subsequence,
\begin{equation}\label{eq:weak-limit-existence}
 \widetilde u_n\rightharpoonup u_0\quad\text{in }X,
 \qquad
 \widetilde u_n\rightarrow u_0\quad\text{a.e. in }\mathbb{R}^N.
\end{equation}
The weak limit is nonzero.  Indeed, boundedness of the weighted term gives, uniformly in $n$,
$$
 \int_{\{|y|>R\}\times Q_0}\widetilde u_n^2\,\mathrm{d}x
 \leq R^{-2}\int_{\mathbb{R}^N}|y|^2\widetilde u_n^2\,\mathrm{d}x
 \leq CR^{-2}.
$$
Choose $R$ large so that the last term is less than $\delta/2$.  Then \eqref{eq:local-mass-positive} leaves at least $\delta/2$ mass in the bounded set $B_R^d\times Q_0$.  Rellich compactness there yields strong $L^2$ convergence, and hence $u_0\not\equiv0$.

It remains to justify passage to the nonlinear limit.  For every compact $K\Subset\mathbb{R}^N$, local Rellich compactness gives
\begin{equation}\label{eq:local-strong-existence}
 \widetilde u_n\rightarrow u_0\quad\text{in }L^r(K)
 \qquad(2\leq r<2^*).
\end{equation}
Fix $\varphi\in C_c^\infty(\mathbb{R}^N)$.  The growth estimate \eqref{eq:g-growth-existence}, together with \eqref{eq:local-strong-existence}, implies that
$\{g(\widetilde u_n)\varphi\}$ is uniformly integrable and converges almost everywhere to $g(u_0)\varphi$.  Vitali's theorem therefore gives
\begin{equation}\label{eq:nonlinear-limit-existence}
 \int g(\widetilde u_n)\varphi\,\mathrm{d}x
 \longrightarrow\int g(u_0)\varphi\,\mathrm{d}x.
\end{equation}
The linear terms pass to the limit by weak convergence.  Since translations in $z$ leave $I$ invariant, \eqref{eq:PS-sequence} and \eqref{eq:nonlinear-limit-existence} imply
$$
 \int\bigl(\nabla u_0\cdot\nabla\varphi+(|y|^2+\lambda)u_0\varphi\bigr)\,\mathrm{d}x
 =\int g(u_0)\varphi\,\mathrm{d}x.
$$
Thus $u_0$ solves \eqref{eq:main} in distributions.  Moreover, \eqref{eq:g-growth-existence} and the Sobolev embedding show that $g(u_0)$ defines a continuous functional on $X$; by density of $C_c^\infty$ in $X$, this is equivalent to $I'(u_0)=0$. For existence, no strong convergence of the entire Palais-Smale sequence is required; the nonzero weak limit already yields a critical point.

Because $u_n\geq0$, we have $u_0\geq0$, and the nonvanishing argument above gives $u_0\not\equiv0$.  We first establish boundedness and local regularity; strict positivity will then follow from the strong maximum principle with a locally bounded coefficient. From \eqref{eq:g-growth-existence},
\begin{equation}\label{eq:subsolution-existence}
 -\Delta u_0
 =-(|y|^2+\lambda)u_0+g(u_0)
 \leq C\bigl(u_0+u_0^{q-1}\bigr).
\end{equation}
For $N\geq3$, the exponent $q-1$ is strictly below $(N+2)/(N-2)$; for $N=2$ it is finite.  The Brezis--Kato iteration, in the form used by Brezis and Lieb, applies to the nonnegative subsolution \eqref{eq:subsolution-existence}; see \cite[Theorem 2.3]{BL-CMP} when $N\geq3$ and \cite[Theorem 3.2]{BL-CMP} when $N=2$.  Although those results are formulated for equations, their iteration uses nonnegative truncations and only the upper differential inequality, so the same estimates hold for \eqref{eq:subsolution-existence}.  Hence $u_0\in L^\infty(\mathbb{R}^N)$.

Since $g(s)/s$ extends to a bounded function on the compact interval $[0,\|u_0\|_\infty]$, \eqref{eq:subsolution-existence} improves to $-\Delta u_0\leq C u_0$.  The uniform local mean-value estimate for nonnegative subsolutions gives
$$
 \sup_{B_1(x_0)}u_0
 \leq C\|u_0\|_{L^2(B_2(x_0))}.
$$
The right-hand side tends to zero as $|x_0|\rightarrow\infty$ because $u_0\in L^2(\mathbb{R}^N)$.  Thus
\begin{equation}\label{eq:decay-existence}
 u_0(x)\longrightarrow0\qquad\text{as }|x|\rightarrow\infty.
\end{equation}

Finally, on every $K\Subset\mathbb{R}^N$ the right-hand side of
$-\Delta u_0=-(|y|^2+\lambda)u_0+g(u_0)$ is bounded.  Interior $W^{2,r}$ estimates for arbitrary finite $r$, followed by Sobolev embedding, yield
$u_0\in C^{1,\alpha}_{\rm loc}$ for every $\alpha\in(0,1)$.  If $g\in C^1$, then the right-hand side is locally H\"older continuous and Schauder estimates give
$u_0\in C^{2,\alpha}_{\rm loc}$ for every $\alpha\in(0,1)$.
At this stage $u_0$ is continuous and the quotient $g(u_0)/u_0$, defined
at zero by its limit $0$, is locally bounded.  Thus $u_0$ is a
nonnegative, nontrivial weak solution of a linear equation with locally
bounded zero-order coefficient.  The strong maximum principle yields
$u_0>0$ in $\mathbb{R}^N$.
This completes the proof.
\end{proof}

\section{\texorpdfstring{Symmetry and strict monotonicity for $\lambda>-d$}{Symmetry and strict monotonicity for lambda>-d}}\label{sec:symmetry}

Throughout this section we assume the hypotheses of Theorem \ref{thm:main-symmetry}.  Thus
$\lambda>-d$, $g\in C^1([0,\infty),[0,\infty))$, $g'(0)=0$, and
$u\in C^2(\mathbb{R}^N)\cap X$ is a positive solution of \eqref{eq:main} satisfying
$u(x)\rightarrow0$ as $|x|\rightarrow\infty$.

We shall repeatedly use the local strong maximum principle and Hopf boundary lemma for linear uniformly elliptic equations with locally bounded zero-order coefficient.  In the applications below the comparison function is already known to be nonnegative, and the coefficient is bounded on every bounded subdomain; the standard localization argument therefore applies without imposing a global sign condition on the zero-order term.  We refer to \cite[Theorem 7.3.3]{Chen-Li-Book}; see also \cite{GidasNiNirenberg1979}.

\subsection{Two preliminary estimates}\label{subsec:sym-prelim}

For $0<\varepsilon<1$, the scaled oscillator inequality from Section \ref{sec:Spectral} gives
\begin{equation}\label{eq:sym-scaled-gap}
 \int_{\mathbb{R}^N}\bigl((1-\varepsilon)|\nabla v|^2+|y|^2v^2\bigr)\,\mathrm{d}x
 \geq d\sqrt{1-\varepsilon}\int_{\mathbb{R}^N}v^2\,\mathrm{d}x,
 \qquad v\in X.
\end{equation}
Since $\lambda>-d$, we may and shall fix $\varepsilon_0\in(0,1)$ so small that
\begin{equation}\label{eq:sym-positive-gap}
 \gamma_0:=d\sqrt{1-\varepsilon_0}+\lambda-\varepsilon_0>0.
\end{equation}

The second estimate is a small-set control for the difference quotient of $g$.  If $M>0$, set
$$
 L_M:=\max_{0\leq s\leq M}|g'(s)|<\infty.
$$
For any $a,b\in[0,M]$ define
\begin{equation}\label{eq:sym-difference-coeff}
 c(a,b):=\int_0^1 g'\bigl(ta+(1-t)b\bigr)\,\mathrm{d}t.
\end{equation}
Then $g(a)-g(b)=c(a,b)(a-b)$, $|c(a,b)|\leq L_M$, and, because $g'(0)=0$,
\begin{equation}\label{eq:sym-tail-c-small}
 \sup_{0\leq a,b\leq\delta}|c(a,b)|\longrightarrow0
 \qquad\text{as }\delta\downarrow0.
\end{equation}
On a measurable set $E$ of finite measure, if $|c|\leq L$ and $v$ is the zero extension of an $H_0^1(E)$ function, then for $N\geq3$,
\begin{equation}\label{eq:sym-small-set-Nge3}
 \int_E |c|v^2\,\mathrm{d}x
 \leq L|E|^{2/N}\|v\|_{2^*}^2
 \leq C L|E|^{2/N}\|\nabla v\|_2^2.
\end{equation}
When $N=2$, fix any $r>2$ and use H\"older and the embedding $H^1(\mathbb{R}^2)\hookrightarrow L^r(\mathbb{R}^2)$ to obtain
\begin{equation}\label{eq:sym-small-set-N2}
 \int_E |c|v^2\,\mathrm{d}x
 \leq C_{r,L}|E|^{1-2/r}
 \bigl(\|\nabla v\|_2^2+\|v\|_2^2\bigr).
\end{equation}
Thus in every dimension considered here the contribution of a bounded coefficient on a set of sufficiently small measure can be absorbed into the left-hand side of \eqref{eq:sym-scaled-gap}.

\subsection{Symmetry in the confined variables}\label{subsec:first-d}

We first move a plane in the $y_1$ direction.  For $\mu\leq0$ let
$$
 \Sigma_\mu:=\{(y,z)\in\mathbb{R}^N:y_1<\mu\},
 \qquad T_\mu:=\partial\Sigma_\mu,
$$
and write
$$
 x^\mu=(2\mu-y_1,y_2,\ldots,y_d,z),
 \qquad u_\mu(x):=u(x^\mu),
 \qquad w_\mu:=u_\mu-u.
$$
For $x\in\Sigma_\mu$,
\begin{equation}\label{eq:sym-potential-order}
 |y|^2-|y^\mu|^2=4\mu(y_1-\mu)\geq0.
\end{equation}
Subtracting the equations for $u_\mu$ and $u$ gives
\begin{equation}\label{eq:sym-w-confined-exact}
 -\Delta w_\mu+(|y|^2+\lambda-c_\mu)w_\mu
 =\bigl(|y|^2-|y^\mu|^2\bigr)u_\mu\geq0
 \quad\text{in }\Sigma_\mu,
\end{equation}
where
\begin{equation}\label{eq:sym-cmu}
 c_\mu(x):=\int_0^1
 g'\bigl(tu_\mu(x)+(1-t)u(x)\bigr)\,\mathrm{d}t.
\end{equation}
This integral definition is valid also at points where $u_\mu=u$ and avoids any ambiguity in the difference quotient.

Let
$$
 \eta_\mu:=\max\{-w_\mu,0\}.
$$
For every fixed $\mu$, reflection is an isometry of the unweighted $H^1$ norm, and
\begin{equation}\label{eq:sym-reflected-weight}
 |y|^2\leq 2|y^\mu|^2+8\mu^2
 \qquad(x\in\Sigma_\mu).
\end{equation}
Hence $u_\mu\in H^1(\Sigma_\mu)$ and $|y|u_\mu\in L^2(\Sigma_\mu)$.  Since $w_\mu=0$ on $T_\mu$, the Lipschitz truncation $\eta_\mu=\max\{-w_\mu,0\}$ belongs to $H_0^1(\Sigma_\mu)$.  On $\{\eta_\mu>0\}$ one has $0<u_\mu<u$ and therefore $0\leq\eta_\mu=u-u_\mu\leq u$; in particular $|y|\eta_\mu\in L^2(\Sigma_\mu)$.  The standard zero-extension theorem for $H_0^1(\Sigma_\mu)$ now shows that the extension of $\eta_\mu$ by zero belongs to $H^1(\mathbb{R}^N)$, and the preceding weighted estimate shows that it actually belongs to $X$.  Equivalently, one may approximate it in the $X$-form norm by truncating away from $T_\mu$ and infinity and then mollifying.  Thus the full-space inequality \eqref{eq:sym-scaled-gap} is legitimately applicable to $\eta_\mu$.

Testing \eqref{eq:sym-w-confined-exact} with $\eta_\mu$ and recalling that $w_\mu=-\eta_\mu$ on its support yields
\begin{equation}\label{eq:sym-basic-negative}
 0\geq
 \int_{\Sigma_\mu}\bigl(|\nabla\eta_\mu|^2+(|y|^2+\lambda)\eta_\mu^2\bigr)\,\mathrm{d}x
 -\int_{\{w_\mu<0\}}c_\mu\eta_\mu^2\,\mathrm{d}x.
\end{equation}
Only the positive part of $c_\mu$ matters in the last term.

\medskip
\noindent\emph{Starting the plane.}
On $\{w_\mu<0\}$ one has $0<u_\mu<u$.  Since $u(x)\rightarrow0$ at infinity, for every $\delta>0$ there is $R>0$ such that $u(x)\leq\delta$ on $|x|\geq R$.  If $\mu<-R$, then every point of $\Sigma_\mu$ lies outside $B_R$; hence \eqref{eq:sym-tail-c-small} gives
$$
 \sup_{\{w_\mu<0\}}(c_\mu)^+\leq o_R(1).
$$
Combining \eqref{eq:sym-basic-negative} with the unscaled oscillator bound
$\int(|\nabla v|^2+|y|^2v^2)\geq d\int v^2$ gives, for $R$ sufficiently large,
$$
 0\geq\bigl(d+\lambda-o_R(1)\bigr)\|\eta_\mu\|_2^2.
$$
Because $d+\lambda>0$, this forces $\eta_\mu=0$.  Thus
\begin{equation}\label{eq:sym-start-confined}
 w_\mu\geq0\quad\text{in }\Sigma_\mu
 \qquad\text{for all sufficiently negative }\mu.
\end{equation}

Define
\begin{equation}\label{eq:sym-mu0-confined}
 \mu_0:=\sup\{\bar\mu\leq0:
 w_\mu\geq0\text{ in }\Sigma_\mu\text{ for every }\mu\leq\bar\mu\}.
\end{equation}
We claim that $\mu_0=0$.  By continuity $w_{\mu_0}\geq0$.  If $\mu_0<0$, then $w_{\mu_0}\not\equiv0$: otherwise, subtracting the two exact equations would give
$\bigl(|y|^2-|y^{\mu_0}|^2\bigr)u=0$ in $\Sigma_{\mu_0}$, contradicting \eqref{eq:sym-potential-order} and $u>0$.  Hence the strong maximum principle gives
\begin{equation}\label{eq:sym-strict-mu0}
 w_{\mu_0}>0\quad\text{in }\Sigma_{\mu_0}.
\end{equation}

We now show that the plane can be moved slightly beyond $\mu_0$.  Choose $R$ so large that, on $B_R^c\cap\{w_\mu<0\}$ and for $\mu$ near $\mu_0$, \eqref{eq:sym-tail-c-small} implies
\begin{equation}\label{eq:sym-tail-absorb}
 (c_\mu)^+\leq\varepsilon_0.
\end{equation}
On the bounded part $B_R$, uniform convergence $w_\mu\rightarrow w_{\mu_0}$ as $\mu\downarrow\mu_0$ and \eqref{eq:sym-strict-mu0} imply
\begin{equation}\label{eq:sym-negative-set-small}
 \bigl|B_R\cap\Sigma_\mu\cap\{w_\mu<0\}\bigr|\longrightarrow0
 \qquad\text{as }\mu\downarrow\mu_0.
\end{equation}
Since $u$ is bounded on the relevant compact set, $|c_\mu|$ is uniformly bounded there.  Applying \eqref{eq:sym-small-set-Nge3} or \eqref{eq:sym-small-set-N2}, and taking $\mu-\mu_0>0$ sufficiently small, we obtain
\begin{equation}\label{eq:sym-c-absorb}
 \int_{\{w_\mu<0\}}(c_\mu)^+\eta_\mu^2\,\mathrm{d}x
 \leq \varepsilon_0\bigl(\|\nabla\eta_\mu\|_2^2+\|\eta_\mu\|_2^2\bigr).
\end{equation}
Putting this into \eqref{eq:sym-basic-negative} and using \eqref{eq:sym-scaled-gap} gives
$$
 0\geq\bigl(d\sqrt{1-\varepsilon_0}+\lambda-\varepsilon_0\bigr)
 \|\eta_\mu\|_2^2
 =\gamma_0\|\eta_\mu\|_2^2.
$$
Thus $w_\mu\geq0$ slightly beyond $\mu_0$, contradicting its definition.  Therefore $\mu_0=0$.

Applying the same argument to $u(-y_1,y',z)$ gives the reverse inequality at $T_0$, hence
\begin{equation}\label{eq:sym-even-y1}
 u(-y_1,y',z)=u(y_1,y',z).
\end{equation}
Rotating the $y$ coordinates and repeating the argument shows invariance under reflection in every hyperplane through the origin in $\mathbb{R}^d$.  Therefore
\begin{equation}\label{eq:sym-y-radial}
 u(y,z)=V(|y|,z).
\end{equation}

The monotonicity is strict.  For each $\mu<0$, $w_\mu\not\equiv0$ by the strict inequality in \eqref{eq:sym-potential-order}; hence $w_\mu>0$ in $\Sigma_\mu$ and the Hopf lemma applies on $T_\mu$.  With the outward normal $\nu=e_1$,
$$
 \partial_\nu w_\mu=-2\partial_{y_1}u<0
 \quad\text{on }T_\mu,
$$
so $\partial_{y_1}u>0$ on $y_1=\mu<0$.  By evenness this is equivalent to
\begin{equation}\label{eq:sym-y-strict}
 \partial_r V(r,z)<0\qquad(r>0).
\end{equation}

\subsection{Symmetry in the free variables and the common center}\label{subsec:last-N-d}

Fix a unit vector $e\in\mathbb{R}^m$.  After an orthogonal change of the $z$ variables we may suppose $e=e_1$ and write $z=(z_1,z')$.  For $\mu\in\mathbb{R}$ set
$$
 \Sigma_\mu^z:=\{(y,z):z_1<\mu\},
 \qquad z^\mu=(2\mu-z_1,z'),
 \qquad u_\mu(y,z):=u(y,z^\mu),
 \qquad w_\mu:=u_\mu-u.
$$
Now the potential $|y|^2$ is invariant under the reflection, so the exact difference equation is
\begin{equation}\label{eq:sym-w-free}
 -\Delta w_\mu+(|y|^2+\lambda-c_\mu)w_\mu=0
 \quad\text{in }\Sigma_\mu^z,
 \qquad w_\mu=0\quad\text{on }\partial\Sigma_\mu^z,
\end{equation}
with $c_\mu$ again given by \eqref{eq:sym-cmu}.

For the negative part $\eta_\mu=\max\{-w_\mu,0\}$, the zero extension belongs to $X$ immediately because reflection in $z$ leaves both $|y|$ and the $X$-norm unchanged.  Thus the same testing identity \eqref{eq:sym-basic-negative} holds.  Since $u$ decays at infinity, the starting argument based on \eqref{eq:sym-tail-c-small} shows
\begin{equation}\label{eq:sym-start-free}
 w_\mu\geq0\quad\text{in }\Sigma_\mu^z
 \qquad\text{for all sufficiently negative }\mu.
\end{equation}
Let
\begin{equation}\label{eq:sym-mu0-free}
 \mu_*(e):=\sup\{\bar\mu\in\mathbb{R}:
 w_\mu\geq0\text{ in }\Sigma_\mu^z\text{ for all }\mu\leq\bar\mu\}.
\end{equation}
The number $\mu_*(e)$ is finite.  Indeed, if the inequality held for every $\mu$, then for every fixed $(y,z')$ the positive function $z_1\mapsto u(y,z_1,z')$ would be nondecreasing on all of $\mathbb{R}$, which is incompatible with $u(y,z_1,z')\rightarrow0$ as $|z_1|\rightarrow\infty$.

At the limiting position $w_{\mu_*(e)}\geq0$.  If it were not identically zero, the strong maximum principle would give $w_{\mu_*(e)}>0$.  The same compact-small-negative-set argument \eqref{eq:sym-negative-set-small}--\eqref{eq:sym-c-absorb}, now with no potential-reflection term, would then permit the plane to move a little farther, contradicting the definition of $\mu_*(e)$.  Hence
\begin{equation}\label{eq:sym-free-plane}
 u(y,z)=u\bigl(y,z-2((z\cdot e)-\mu_*(e))e\bigr),
\end{equation}
that is, $u$ is symmetric with respect to the hyperplane
$$
 P_e:=\{z\in\mathbb{R}^m:z\cdot e=\mu_*(e)\}.
$$
Moreover, the comparison is strict at every position $\mu<\mu_*(e)$.  Indeed,
if $w_\mu\equiv0$ for some such $\mu$, choose
$\nu\in(\mu,\mu_*(e))$ and write, on a line parallel to $e$,
$f(t)=u(y,z'+te)$.  Symmetry about $t=\mu$ and the inequality associated
with the plane $t=\nu$ imply, for every $t\geq\mu$,
$$
 f\bigl(t+2(\nu-\mu)\bigr)\geq f(t).
$$
Iteration contradicts $f(t)\rightarrow0$ as $t\rightarrow+\infty$, since $f(t)>0$.
Thus $w_\mu\not\equiv0$, and the strong maximum principle gives
$w_\mu>0$ before the limiting position.  The Hopf lemma therefore gives
strict monotonicity along every line orthogonal to $P_e$, away from the
plane.

It remains to prove that the hyperplanes $P_e$ have a common center.  Apply the preceding construction to an orthonormal basis $e_1,\ldots,e_m$ and put
\begin{equation}\label{eq:sym-z0}
 z_0:=\sum_{j=1}^m\mu_*(e_j)e_j.
\end{equation}
For each fixed $y$, strict monotonicity in the $e_j$ direction implies successively that
\begin{equation}\label{eq:sym-unique-z-max}
 u(y,z)<u(y,z_0)\qquad\text{whenever }z\ne z_0.
\end{equation}
Thus $z_0$ is the unique maximum point of the map $z\mapsto u(y,z)$, for every fixed $y$.

Now take an arbitrary unit vector $e$.  Reflection in $P_e$ preserves $u$.  If $z_0\notin P_e$, its reflection $z_0^e$ across $P_e$ would satisfy $z_0^e\ne z_0$ and
$$
 u(y,z_0^e)=u(y,z_0)
 \qquad\text{for every }y,
$$
contradicting \eqref{eq:sym-unique-z-max}.  Hence $z_0\in P_e$ for every $e$, or equivalently
\begin{equation}\label{eq:sym-plane-through-center}
 \mu_*(e)=z_0\cdot e\qquad(e\in\mathbb S^{m-1}).
\end{equation}
Therefore $u$ is invariant under reflection in every hyperplane of $\mathbb{R}^m$ passing through $z_0$.  These reflections generate the full orthogonal group about $z_0$, and we conclude that
\begin{equation}\label{eq:sym-z-radial}
 u(y,z)=U(|y|,|z-z_0|).
\end{equation}
The strict one-dimensional monotonicity obtained above along every direction through $z_0$ yields
\begin{equation}\label{eq:sym-z-strict}
 \partial_sU(r,s)<0\qquad(s>0).
\end{equation}
Together with \eqref{eq:sym-y-strict}, this proves Theorem \ref{thm:main-symmetry}.

\subsection*{Data availability}
No data was used for the research described in the article.

\subsection*{Acknowledgments}
Z. Chen was supported by National Key R\&D Program of China (No.~2023YFA1010002) and NSFC (No.~12222109). J. Zhang was supported by the NSFC (No.~12371109). X. Zhong was supported by the NSFC (No.~12271184), Young Top-notch Talent Project of Guangdong Province (No.~2024TQ08A725), and Guangzhou Basic and Applied Basic Research Foundation (No.~2024A04J10001).

\end{document}